\documentclass[11pt,amssymb]{amsart}
\usepackage{amssymb}
\usepackage{hyperref}
\usepackage[bottom]{footmisc}
\usepackage{enumitem}
\usepackage[mathscr]{eucal}
\usepackage[all,cmtip]{xy}
\usepackage{amscd}
\usepackage{mathtools}
\usepackage{tikz-cd}

\newcommand{\im}{{\rm im}\:}

\newcommand{\Q}{{\mathbb Q}}

\newtheorem{thm}{Theorem}[section]

\newtheorem{lemma}[thm]{Lemma}
\newtheorem{prop}[thm]{Proposition}
\newtheorem{cor}[thm]{Corollary}

\theoremstyle{definition}
\newtheorem{definition}[thm]{Definition}
\newtheorem{example}[thm]{Example}
\newtheorem{remark}[thm]{Remark}

\newcommand{\coker}{\text{coker }}

\newcommand{\QQ}{\mathbb{Q}}

\newcommand{\AF}{\mathbb{A}}
\newcommand{\II}{\mathbb{I}}

\newcommand{\ZZ}{\mathbb{Z}}

\newcommand{\GG}{\mathbb{G}}

\newcommand{\mc}[1]{\mathcal{#1}}

.240pk scaled 1200 .240pk
\usepackage[all,cmtip]{xy}
\begin{document}

\title[On almost strong approximation in reductive algebraic groups]{On almost strong approximation in reductive algebraic groups}

\author[A.S.~Rapinchuk]{Andrei S. Rapinchuk}

\author[W.~Tralle]{Wojciech Tralle}

\address{Department of Mathematics, University of Virginia,
Charlottesville, VA 22904-4137, USA}

\email{asr3x@virginia.edu}

\address{Department of Mathematics, University of Tampa,
Tampa, FL 33606, USA}

\email{wtralle@ut.edu}

\begin{abstract}
We investigate a slight weakening of the classical property of strong approximation, which we call {\it almost strong approximation}, for connected reductive algebraic groups over global fields with respect to special sets of valuations. While nonsimply connected groups (in particular, all algebraic tori) always fail to have strong approximation -- and even almost strong approximation -- with respect to any finite set of valuations, we show that under appropriate assumptions they do have almost strong approximation with respect to certain infinite sets of valuations that can be characterized in terms of Dirichlet density and include {\it tractable} sets of valuations, i.e. those sets that contain all archimedean valuations and a generalized arithmetic progression minus a set of Dirichlet density zero. Almost strong approximation is likely to have a variety of applications, and as an example we use almost strong approximation in tori with respect to tractable sets to extend the essential part of the result of Radhika and Raghunathan  (cf. \cite[Theorem 5.1]{RadRag}) on the congruence subgroup problem for inner forms of type $\textsf{A}_n$ to all absolutely almost simple simply connected groups.  
\end{abstract}

\maketitle

\section{Introduction}\label{Introduction}

The goal of this paper is to develop new results on strong approximation for reductive algebraic groups over global fields and apply them to the congruence subgroup problem. While strong approximation in semi-simple {\it simply connected} groups has been one of the main tools in the arithmetic theory of algebraic groups since its inception, here we establish (under appropriate assumptions) a slightly weaker property that we term {\it almost strong approximation} for (connected reductive) {\it nonsimply connected groups} where strong approximation in the classical situation never holds. 

More precisely, let $G$ be a connected linear algebraic group 
defined over a global field $K$. Given a~nonempty subset $S$ of the set $V^K$ of all inequivalent valuations of $K$, we let $\mathbb{A}_{K,S}$ denote the {\it ring of $S$-adeles} of $K$, i.e. the ring of truncated adeles without the components corresponding to valuations $v \in S$, or more explicitly, the restricted topological product of the completions $K_v$ for $v \in V^K \setminus S$ with respect to the valuation rings $\mathcal{O}_v \subset K_v$ for nonarchimedean $v \in V^K \setminus S$.   Then the {\it group of $S$-adeles} of $G$ can be identified with the group of points $G(\mathbb{A}_{K,S})$, equipped with the {\it $S$-adelic topology} (see \cite[\S 5.1]{PlRa} or \cite{Ra-SA} for basic definitions).  The group of $K$-rational points $G(K)$ admits a diagonal embedding $G(K) \hookrightarrow G(\mathbb{A}_{K,S})$, and is typically identified with the image of this embedding  called the {\it group of principal adeles}. One says that $G$ has {\it strong approximation} with respect to $S$ if the diagonal embedding is dense, in other words, $\overline{G(K)}^{(S)} = G(\mathbb{A}_{K,S})$ where $\bar{\ }\bar{\ }^{(S)}$ denotes the closure in the $S$-adelic topology\footnote{For the context, we recall that the diagonal embedding $K \hookrightarrow \mathbb{A}_{K,S}$ has discrete image if $S = \varnothing$, and dense image for any nonempty $S$ (cf. \cite[Ch. II]{ANT}). Thus, no nontrivial linear algebraic group $G$ can have strong approximation for $S$ empty -- which is the reason why we always assume that $S \neq \varnothing$ in this paper,  while the additive group $G = \mathbb{G}_a$ does have strong approximation for every nonempty $S$.}. Strong approximation in algebraic groups has been studied extensively since 1930s in the works of M.~Eichler, M.~Kneser, and others. For {\it finite} $S$, a criterion for strong approximation in reductive groups was obtained first by V.P.~Platonov \cite{Pl-SA} in characteristic zero and later by G.A.~Margulis \cite{M-SA} and G.~Prasad \cite{Pr-SA} in positive characteristic. The ``only if'' part of this criterion implies that $G$ never has strong approximation for any finite $S$ unless it is simply connected, and in fact in the nonsimply connected case the index $[G(\mathbb{A}_{K,S}) : \overline{G(K)}^{(S)}]$ is always infinite (cf. \cite[Prop. 7.13]{PlRa}).  This applies, in particular, to any nontrivial algebraic $K$-torus $T$, where one can actually show that the quotient $T(\mathbb{A}_{K,S})/\overline{T(K)}^{(S)}$ is a group of infinite exponent for any finite $S$ (cf. \cite[Proposition 2.1]{Ra-SA}). 

On the other hand, it was shown in \cite[Theorem 5.3]{PrRa} (see also \cite[Theorem 3]{PrRa-Irr}) that the exponent of the quotient $T(\mathbb{A}_{K,S})/\overline{T(K)}^{(S)}$ becomes finite if $S$ contains the set $V^K_{\infty}$ of archimedean valuations  as well as  {\it all but finitely many} valuations in a generalized arithmetic progression (see the definition below) that satisfies one technical condition (which cannot be omitted). This result was applied in \cite[\S 5]{PrRa} to the Congruence Subgroup Problem (CSP). The starting point for the current paper was the project to obtain similar results for more general sets of valuations we refer to as {\it tractable} that were suggested by the work of Radhika and Raghunathan  \cite{RadRag} and extend the result on CSP in {\it loc. cit.} to all types of absolutely almost simple simply connected algebraic groups. This project was completed in early 2025 and the results were posted on arXiv \cite{RaTr}. Recently, Y.~Cao and Y.~Wang \cite{CW} observed that our results on strong approximation can be extended to sets that can be characterized in terms of Dirichlet density. The proofs in \cite{CW} rely on the results of Demarche \cite{Dem1}, \cite{Dem2} and use Brauer-Manin obstruction. We will show that improvements proposed in \cite{CW} can be achieved by one relatively minor change in the original preprint \cite{RaTr}: one just needs to replace Proposition 3.1 in \cite{RaTr} (which is reproduced unaltered in \S \ref{ASAquasisplit} below) with Proposition \ref{P:1000} inspired by \cite{CW}, without any changes in the rest of the arguments -- we note that our approach requires a more modest input limited basically to the Nakayama-Tate theorem. To make our account more comprehensive, in the formulations of the main results we include the more general statements in terms of Dirichlet density as well as the initial more concrete statements for tractable sets that will actually be used in applications to CSP in \S \ref{CSPapplication}.

The results of the paper assert that under appropriate assumptions the index $[G(\mathbb{A}_{K,S}) : \overline{G(K)}^{(S)}]$ is finite. Simple examples (cf. Example \ref{exampleASAindex}) demonstrate that while finite, this index is in general $> 1$ even when $G$ is a split torus over $K = \mathbb{Q}$ and $S$ consists of the archimedean valuation and all primes in an arithmetic progression, and thus $G$ fails to have strong approximation in the classical sense. This suggests the following new notion: we say that an algebraic $K$-group $G$ has {\it almost strong approximation}\footnote{As the referee pointed out, the term ``almost strong approximation" was already used in \cite{Hsia1} and \cite{Hsia2} to describe a completely different property. Indeed, the focus in \cite{Hsia1} and \cite{Hsia2} was on analyzing the weakening of strong approximation for $G$ with respect to $S$ when the group $G_S = \prod_{v \in S} G(K_v)$ is compact, violating thereby one of the necessary conditions for strong approximation. In the current paper, we analyze the weakening of strong approximation when the other necessary condition, viz. simply connectedness, fails. So, the use of the same terminology seems appropriate and should not lead to a confusion, although the property we investigate can also be adequately described by a different term like ``virtual strong approximation." } with respect to a subset $S \subset V^K$ if the index $[G(\mathbb{A}_{K,S}) : \overline{G(K)}^{(S)}]$ is finite -- see Definition \ref{D:ASA} and the discussion that follows for relations between almost strong approximation and honest/classical strong approximation. In order to give precise statements, we need the following definitions. 

\vskip1mm

%In the current paper we consider more general sets $S$ that contain $V^K_{\infty}$ and a generalized arithmetic progression minus an arbitrary set having Dirichlet density zero -- we call such sets {\it tractable}.  We then prove that for a $K$-torus $T$ and any tractable set $S$ the quotient $T(\mathbb{A}_K(S))/\overline{T(K)}^{(S)}$ is in fact {\it finite} provided that the same technical condition as in \cite{PrRa} holds for the generalized arithmetic progression involved in the description of $S$. Since one cannot guarantee that this quotient is trivial (which would mean that $T$ has strong approximation with respect $S$ in the classical sense) in a sufficiently general situation, we introduce the following new notion: we say that an algebraic $K$-group $G$ has {\it almost strong approximation} with respect to a subset $S \subset V^K$ if the index $[G(\mathbb{A}_K(S)) : \overline{G(K)}^{(S)}]$ is finite (cf. Definition \ref{D:ASA}).  After proving the above results on almost strong approximation in tori, we extend them to arbitrary reductive groups. 
%
%\vskip1mm 
%
%In order to give precise statements, we need the following definitions. 
 
\begin{definition}
Let $L/K$ be a finite Galois extension and let $\mc{C}$ be a conjugacy class in the Galois group $\mathrm{Gal}(L/K)$. 
A {\it generalized arithmetic progression} $\mathcal{P}(L/K , \mathcal{C})$ is the set of all $v \in V^K_f := V^K \setminus V^K_{\infty}$ such that $v$ is unramified in $L$ and for some (equivalently, any) extension $w | v$ the corresponding Frobenius automorphism $\mathrm{Fr}_{L/K}(w | v)$ lies in $\mathcal{C}$.
\end{definition}

In this paper, we will not differentiate between (finite) primes of a global field and the corresponding nonarchimedean valuations. Under this convention, the generalized arithmetic progression $\mathcal{P}(L/\mathbb{Q} , \mathcal{C}_a)$, where $L = \Q(\zeta_m)$ is the $m$th cyclotomic extension of $\mathbb{Q}$ and $\mathcal{C}_a$ with $(a , m) = 1$ consists 
of the automorphism $\sigma_a \in \mathrm{Gal}(L/K)$ defined by $\sigma_a(\zeta_m) = \zeta_m^a$ coincides with the set of rational primes $p$ satisfying $p \equiv a(\mathrm{mod}\: m)$
(cf., for example, \cite[Ch. VII, \S 3, Example 3.4]{ANT}), hence the terminology.

We refer the reader to \cite[Ch. VII, \S 13]{Neu} or \cite[Ch. IV, \S 4]{Jan} (see also \S  \ref{ASAquasisplit} below) for the notion of Dirichlet density $\mathfrak{d}_K(\mathcal{P})$
of any set $\mathcal{P}$ of primes of $K$ and its basic properties. Here we only recall that finite sets of primes have Dirichlet density zero, however it is easy to construct infinite sets of primes with density zero for any $K$, and in fact for a given nontrivial finite extension $K/\mathbb{Q}$ the set of all primes of $K$ that have relative degree $> 1$ over $\mathbb{Q}$ has density zero (cf. \cite[4.7.1]{Jan}). 

\begin{definition}
A subset $S \subset V^K$ is {\it tractable} if it contains a set of the form $V^K_{\infty} \cup (\mathcal{P}(L/K , \mathcal{C}) \setminus \mathcal{P}_0)$ for some generalized arithmetic progression $\mathcal{P}(L/K , \mathcal{C})$ and a subset $\mathcal{P}_0$ with $\mathfrak{d}_K(\mathcal{P}_0) = 0$. 
\end{definition}

Finally, for a finite (separable) extension $F/K$ of global fields, we let $\mathrm{Spl}(F/K)$ denote the set of nonarchimedean valuations $v$ of $K$ that split completely in $F$ (cf. \cite[p. 49]{Neu}).  

Here is our main result on almost strong approximation in  connected reductive groups. Before formulating it, we recall that the (Galois) twists of a split semi-simple group by (Galois) cocycles with values in the group of {\it inner} automorphisms of the latter are called {\it inner forms}. Then for any semi-simple group there exists the minimal extension of the base field (which is automatically a finite Galois extension) over which the group becomes an inner form of a split group. Furthermore, we assume that all relevant finite extensions of the base field are 
contained in its fixed algebraic closure. 
\begin{thm}\label{T:A}
For a connected reductive algebraic group $G$ defined over a number field $K$, we  let $T = Z(G)^{\circ}$ (resp., $H = [G , G]$) denote the maximal central torus (resp., the maximal semi-simple subgroup) so that $G = TH$ is an almost direct product. Set $E=PM$, where $P/K$ is the minimal splitting field of $T$ and $M/K$ is the minimal Galois extension over which $H$ becomes an inner form of a $K$-split group. If $S \subset V^K$ contains $V^K_{\infty}$ and satisfies $\mathfrak{d}_K(S \cap \mathrm{Spl}(E/K)) > 0$ then the closure $\overline{G(K)}^{(S)}$ is a finite index normal subgroup of $G(\AF_{K,S})$ with abelian quotient, and thus $G$ has almost strong approximation with respect to $S$. 

Furthermore, if $S$ is a tractable set of valuations containing a set of the form $V^K_{\infty} \cup (\mc{P}(L/K,\mc{C})\setminus\mc{P}_0)$ such that for some $\sigma \in \mathcal{C}$, we have 
\begin{equation}\label{technicalconditionreductive}
\sigma|(E\cap L)=\mathrm{id}_{E\cap L},
\end{equation}
the order of the abelian group  $G(\AF_{K,S})/\overline{G(K)}^{(S)}$ divides a constant $C(\ell, n, r)$ that depends only on $\ell = $ rank of $G$, $n=[L:K]$, and $r = $ number of real valuations of $K$.  
\end{thm}

It should be noted that a tractable set $S$ as in the statement of the theorem does satisfy the condition $\mathfrak{d}_K(S \cap \mathrm{Spl}(E/K)) > 0$ -- see the discussion after Proposition \ref{P:1000}. On the other hand, if $G$ is a semi-simple $K$-group which is an inner twist of a $K$-split group, then $E = K$ and hence the latter condition reduces to $\mathfrak{d}_K(S \cap V^K_f) > 0$. So, we obtain the following. 

\begin{cor}\label{C:1}
Let $G$ be a semi-simple $K$-group which is an inner form of a $K$-split group. Then $G$ has almost strong approximation with respect to any  set $S \subset V^K$ containing $V^K_{\infty}$ and satisfying $\mathfrak{d}_K(S \cap V^K_f) > 0$ and in particular for any tractable set. 
\end{cor}

However, it is important to note that condition (\ref{technicalconditionreductive})  on tractable sets in the theorem cannot be omitted in the general case. To demonstrate this, in \S \ref{ASAexamples} (resp., in \S \ref{ASAfailure}) we construct an example of a torus (resp., a nonsimply connected absolutely almost simple  group) that fails to have almost strong approximation for a tractable set $S$ that does not satisfy (\ref{technicalconditionreductive}). 

The proof of Theorem \ref{T:A} involves several steps that rely on different techniques. First,  using results of class field theory and the Chebotarev Density Theorem, we handle the case of quasi-split tori (see Theorem \ref{T:2}). We then realize an arbitrary $K$-torus $T$ as a quotient of a quasi-split one and prove almost strong approximation for $T$ (under appropriate assumptions) using 
the Nakayama-Tate Theorem. It turns out that when $S$ is a tractable set and the condition 
(\ref{technicalconditionreductive}) holds, the quotient $T(\mathbb{A}_{K,S})/\overline{T(K)}^{(S)}$ is a finite group  of order dividing a constant $\tilde{C}(d , n)$ that depends only on $d = \dim T$ and $n = [L : K]$, see Theorem \ref{T:1} and (\ref{C(d,n)}).  

For an arbitrary reductive $K$-group $G=TH$ having nontrivial semi-simple part $H$, we first consider the special case where $H$ is simply connected -- see Proposition \ref{P:simp-con}, and then reduce the general case to the special case by using constructions and techniques from the theory of algebraic groups. 

\vskip6pt

Being available for not necessarily simply connected groups, almost strong approximation is likely to expand the range of applications of the classical property of strong approximation, many of which can be found in \cite{PlRa}. As an example of a new application, we present here a result on the {\it Congruence Subgroup Problem}. We refer the reader to \S 8.1 for a discussion of the CSP, including  the notion of the congruence kernel, the statement of Serre's Congruence Subgroup Conjecture and a survey of some of the available results. Among recent developments one can mention a uniform (i.e., requiring no case-by-case considerations) proof \cite{PrRa} of the triviality of the congruence kernel $C^S(G)$ for an absolutely almost simple simply connected algebraic group $G$ over a global field $K$ with respect to a set of valuations $S \subset V^K$ that contains $V^K_{\infty}$ and contains all but finitely many valuations in a generalized arithmetic progression but does not contain any nonarchimedean valuation $v$ such that $G$ is $K_v$-anisotropic -- among other things, this result provides additional evidence for Serre's conjecture. Subsequently, Radhika and Raghunathan \cite{RadRag} showed that the result remains valid for anisotropic inner forms of type $\textsf{A}_n$ (i.e., for norm one groups $G = \mathrm{SL}_{1 , D}$ associated with central division $K$-algebras) for a larger class of sets $S$ that basically coincides with our tractable sets. Using our results on almost strong approximation for tori with respect to tractable sets, we were able to extend the result of \cite{RadRag} to groups of all types.

\begin{thm}\label{T:B}
Let $G$ be an absolutely almost simple simply connected algebraic group defined over a global field $K$, and let $M/K$ be the minimal Galois extension over which $G$ becomes an inner form of a $K$-split group. Assume that the Margulis-Platonov conjecture (MP) (cf. \S \ref{CSPapplication}) holds for $G(K)$. Let $S\subset V^K$ be a tractable set containing a set of the form  $V^K_{\infty} \cup (\mathcal{P}(L/K , \mathcal{C}) \setminus \mathcal{P}_0)$, where 
\begin{equation}\label{E:XXX}
\sigma \vert (M\cap L) = \mathrm{id}_{M \cap L} \ \ \text{for some} \ \  \sigma \in \mathcal{C},
\end{equation}
and does not contain any nonarchimedean $v$ for which $G$ is $K_v$-anisotropic. Then the congruence kernel $C^S(G)$ is trivial.
\end{thm}

It should be noted that in the cases where Serre's congruence subgroup conjecture remains open, this theorem is the best result available at this point. As we already pointed out, for inner forms of type $\textsf{A}_n$, Theorem \ref{T:B} is due to Radhika and Raghunathan, and now we would like to transcribe it for outer forms of this type.
 
\begin{cor}\label{corB}
Let $G$ be an absolutely almost simple simply connected outer form of type $\mathsf{A}_{\ell}$, i.e. $G = \mathrm{SU}_n(D,h)$, the special unitary group of an $n$-dimensional non-degenerate $\tau$-hermitian form $h$ over a division algebra $D$ of degree $d$ whose center $M$  is a quadratic extension of $K$ and the involution $\tau$ of $D$ satisfies $M^{\tau} = K$, with $\ell = dn - 1$. Assume that (MP)  holds for  $G(K)$.
Let $S \subset V^K$ be a tractable set containing a set of the form  $V^K_{\infty} \cup (\mathcal{P}(L/K,\mc{C}) \setminus \mathcal{P}_0)$  where $\sigma \vert (M\cap L) = \mathrm{id}_{M \cap L}$ for some   $\sigma \in \mathcal{C}$, 
and does not contain any nonarchimedean $v$ for which $G$ is $K_v$-anisotropic. Then $C^S(G) =  1 $. 
\end{cor}

Some comments are in order. Since (MP) is known for inner forms of type $\textsf{A}_n$ (see \cite{R-MP}, \cite{Segev}), no assumption on the truth of (MP) was needed in \cite{RadRag}. Second, since \cite{RadRag} deals only with inner forms, (\ref{E:XXX}) holds automatically in their situation, while in the general case we need to assume (\ref{E:XXX}) in order to apply our results on almost strong approximation with respect to tractable sets. Of course, one can expect to generalize Theorem \ref{T:B} to arbitrary subsets $S \subset V^K$ containing $V^K_{\infty}$, not containing any nonarchimedean anisotropic places for $G$ and satisfying $\mathfrak{d}_K(S \cap \mathrm{Spl}(M/K)) > 0$.    

\vskip1mm 

The structure of the paper is as follows. Section \ref{ASAexamples} contains some motivating  examples that lead to the adequate  statement of the theorem on almost strong approximation in tori -- cf. Theorem \ref{T:1}. The proof of this theorem is  carried out in \S \ref{ASAquasisplit} for quasi-split tori by first working out the case of the multiplicative group of a finite separable extension by means of  techniques from class field theory and Dirichlet density considerations. Then, after developing in \S \ref{BoundH1} some auxiliary cohomological results based on the Nakayama-Tate theorem, we complete the proof of Theorem \ref{T:1} in \S \ref{ASAtori}. Almost strong approximation in general reductive groups is analyzed in \S \ref{ASAreductive} where we prove our main Theorem \ref{T:A}. Then in \S \ref{ASAfailure} we construct an example of an absolutely simple adjoint group, which is an outer form of a split group, that does {\it not} have almost strong approximation with respect to a tractable set of valuations for which the condition (\ref{technicalconditionreductive}) fails. Finally, in \S \ref{CSPapplication} we review the required material dealing with the congruence subgroup problem, summarize the approach to proving the triviality of the congruence kernel developed in \cite{PrRa}, and then apply our Theorem \ref{T:1} to prove Theorem \ref{T:B}.
 
\vskip3pt

\textbf{Notations and conventions.}
Throughout the paper, $K$ will denote a global field. We let $V^K$ denote the set of all inequivalent nontrivial valuations of $K$, with $V^K_{\infty}$ and $V^K_f$ denoting the subsets of archimedean and nonarchimedean valuations, respectively. For $v \in V^K$, the corresponding {\it completion} will be denoted $K_v$, and for $v \in V^K_f$ the {\it valuation ring} of the latter will be denoted $\mathcal{O}_v$ (occasionally, we will set $\mathcal{O}_v = K_v$ for $v \in V^K_{\infty}$). Given a subset $S \subset V^K$, we let $\mathbb{A}_{K,S}$ denote the {\it ring of $S$-adeles} of $K$ (cf. \cite[Ch. 1, \S1.2.1]{PlRa2}), and then for a linear algebraic $K$-group $G$ we set $G(\mathbb{A}_{K,S})$ to be the group of $\mathbb{A}_{K,S}$-points of $G$, in other words, the {\it group of $S$-adeles} of $G$ equipped with the adelic topology (cf. \cite[Ch. 5, \S5.1]{PlRa2}). As usual, $\mathbb{G}_m$ denotes the 1-dimensional split torus, and then for $F/K$ a finite separable extension, $\mathrm{R}_{F/K}(\mathbb{G}_m)$ denotes the corresponding {\it quasi-split torus}, with $\mathrm{R}^{(1)}_{F/K}(\mathbb{G}_m)$ reserved to denote the associated {\it norm torus}. Finally, we will use $\mathbb{P}$ to denote the set of all rational primes, and for integers $a , m$ with $(a , m) = 1$ the symbol $\mathbb{P}_{a(m)}$ will represent the set of all primes $p$ satisfying $p \equiv a(\mathrm{mod}\: m)$ (i.e., all primes in the corresponding {\it arithmetic progression}). 

\vskip3pt 

\textbf{Acknowledgements.} We would like to thank the referees for their comments and suggestions.

\section{Almost strong approximation for tori: motivating examples and the statement of the main theorem}\label{ASAexamples}

We begin with some examples involving tori that, on the one hand, motivate the property of almost strong approximation and indicate that it can be expected to hold for arithmetic progressions, and, on the other hand, exhibit some constraints. Our main result for tori (Theorem \ref{T:1}) shows that almost strong approximation does hold not only for tractable sets associated with generalized arithmetic progressions subject to some constraints of this kind but for more general sets characterized in terms of Dirichlet density.

Let $K$ be a global field, and let $S \subset V^K$ be a nonempty subset. First, we consider the case of the 1-dimensional split torus $T = \mathbb{G}_m$ where the corresponding group of $S$-adeles $T(\mathbb{A}_{K,S})$ coincides with the group of $S$-ideles $\mathbb{I}_{K,S}$ of $K$ (cf. \cite[Ch. II]{ANT}). To make our considerations explicit, we recall that $\mathbb{I}_{K,S}$ is the restricted topological product of the multiplicative groups $K_v^{\times}$ of the completions $K_v$ for $v \in V^K \setminus S$ with respect to the open subgroups $\mathcal{O}_v^{\times} \subset K_v^{\times}$ (with the convention above that $\mathcal{O}_v = K_v$ if $v \in V^K_{\infty}$). Then algebraically $\mathbb{I}_{K,S}$ coincides with the group of invertible elements $\mathbb{A}_{K,S}^{\times}$ but the natural $S$-idelic topology on $\mathbb{I}_{K,S}$ is {\it stronger} than topology induced by the $S$-adelic topology on $\mathbb{A}_{K,S}$ and admits a basis of open sets consisting of sets of the form 
\begin{equation}\label{E:basis}
\prod_{v \in S'} W_v \times \prod_{v \notin S' \cup S} \mathcal{O}_v^{\times}, 
\end{equation}
where $S' \subset V^K \setminus S$ is an arbitrary finite subset, and $W_v \subset K_v^{\times}$ are arbitrary open subsets for $v \in S'$. Equipped with this topology, $\mathbb{I}_{K,S}$ is a locally compact topological group with a distinguished open subgroup 
$$
\mathbb{U}_S := \prod_{v \notin S} \mathcal{O}_v^{\times}. 
$$
It follows from \cite[1.2.1]{PlRa2} that for a number field $K$ and $S = V^K_{\infty}$, the index $[\mathbb{I}_{K,S} : \mathbb{U}_{S}K^{\times}]$ (with $K^{\times}$ embedded in $\mathbb{I}_{K,S}$ diagonally) equals the class number $h(K)$ of $K$, hence is always finite. Thus, for any $S \supset V^K_{\infty}$ the index $[\mathbb{I}_{K,S} : \mathbb{U}_SK^{\times}]$ is finite, and is equal to one if $h(K) = 1$. 
In this case, the intersection $\mathbb{U}_S \cap K^{\times}$ is precisely the group of $S$-units $\mathbb{E}_S$ in $K$, and as $\mathbb{U}_S$ is open in $\mathbb{I}_{K,S}$, we conclude that the index of the closure $[\mathbb{I}_{K,S} : \overline{K^{\times}}^{(S)}]$ is finite if and only if the index $[\mathbb{U}_S : \overline{\mathbb{E}_S}^{(S)}]$ is finite (here  $\bar{\  }\bar{\ }^{(S)}$ denotes the closure in the $S$-idelic topology). Moreover, if $\overline{K^{\times}}^{(S)} = \mathbb{I}_{K,S}$ then $\overline{\mathbb{E}_S}^{(S)} = \mathbb{U}_S$, and the converse is true if $h(K) = 1$. Similar remarks are valid also in the function field case, but we will not formulate them here since in our examples we will stick to the number field case. 

After these preliminaries, we are ready to test strong approximation in some cases. If $K=\QQ$, we simplify notation by writing $\II_S$ (resp. $\AF_S$) rather than $\II_{\QQ,S}$ (resp. $\AF_{\QQ,S}$) for any subset $S\subset V^\QQ$, and also write $\mathbb{I}$ to denote the full group of ideles of $\mathbb{Q}$. First, let $K = \mathbb{Q}$ and $S =\{v_\infty\}$, where $v_\infty$ denotes the unique archimedean valuation of $\QQ$. Then $\mathbb{E}_S = \{ \pm 1 \}$, and hence the index $[\mathbb{U}_S : \overline{\mathbb{E}_S}^{(S)}]$ is infinite (in fact, uncountable), hence the index $[\mathbb{I}_S : \overline{\QQ^{\times}}^{(S)}]$ is also infinite. Now, let $S = \{v_{\infty}\} \cup \{ v_2 \}$ where $v_2$ is the dyadic valuation of $K = \mathbb{Q}$, in which case $\mathbb{E}_S = \langle -1 , 2 \rangle$ is already infinite. Set $Q = \mathbb{P}_{1(8)}$, which is infinite by Dirichlet's Theorem on arithmetic progressions. For every $q \in Q$ we have $\mathbb{E}_S \subset {\mathbb{Z}_q^{\times}}^2$; in other words, $\mathbb{E}_S$ is contained in the kernel of the canonical continuous surjective homomorphism 
\begin{equation}\label{E:square}
\mathbb{U}_S \longrightarrow \prod_{q \in Q} \mathbb{Z}_q^{\times} / {\mathbb{Z}_q^{\times}}^2, 
\end{equation}
implying that the indices $[\mathbb{U}_S : \overline{\mathbb{E}_S}^{(S)}]$, hence also $[\mathbb{I}_S : \overline{\QQ^{\times}}^{(S)}]$, are infinite. It is easy to see that this result extends to any subset $S$ of the form $S = \{v_{\infty}\} \cup \{v_{p_1}, \ldots , v_{p_r}\}$ for any finite collection of primes $p_1, \ldots , p_r$: one simply needs to take $Q = \mathbb{P}_{1(4p_1 \cdots p_r)}$ in the above argument\footnote{This argument can be extended even further to show that a nontrivial torus $T$ over a global field $K$ cannot have strong approximation with respect to any finite set $S \subset V^K$ using such ingredients as Dirichlet Unit Theorem and Chebotarev Density Theorem -- cf. \cite[2.2]{Ra-SA} for details in the number field case.}. On the other hand, an argument of this type cannot be implemented whenever $S$ is infinite, which raises the question if $T = \mathbb{G}_m$ always has strong approximation for $S$ infinite. The following example shows that this is not the case. 

\begin{example}\label{infSnotarithmetic}
\rm{For a prime $p >2$ and any integer $x$ not divisible by $p$, we let $\displaystyle \left( \frac{x}{p} \right)$ denote the corresponding Legendre symbol. As above, it is enough to construct two {\it infinite} disjoint subsets $P = \{p_1, p_2, \ldots \}$ and $Q = \{q_1, q_2, \ldots \}$ of $\mathbb{P}_{1(4)}$ such that 
\begin{equation}\label{E:Leg}
\left( \frac{p}{q} \right) = 1 \ \ \text{for all} \ \ p \in P, \ q \in Q. 
\end{equation}
Indeed, then for $S = \{v_{\infty}\} \cup \{ v_p \, \vert \, p \in P\}$, the group $\mathbb{E}_S$, which is generated by $-1$ and all primes $p \in P$, is contained in the kernel of the map (\ref{E:square}), making the index $[\mathbb{I}_S : \overline{\QQ^{\times}}^{(S)}]$ infinite. 

We construct the required sets $P$ and $Q$ inductively. Pick an arbitrary $p_1 \in \mathbb{P}_{1(4)}$ (e.g., one can take $p_1 = 5$) and choose $q_1 \in \mathbb{P}_{1(4)}$ so that $q_1 \equiv 1(\mathrm{mod}\: p_1)$ (e.g., take $q_1 = 11$). Then using quadratic reciprocity we obtain $$ \left( \frac{p_1}{q_1}\right ) = \left( \frac{q_1}{p_1}\right) = \left( \frac{1}{p_1} \right) = 1.$$ Suppose that we have already found $p_1, \ldots , p_{\ell}$ and $q_1, \ldots , q_{\ell}$ $(\ell \geq 1)$ such that 
$$
\left( \frac{p_i}{q_j} \right) = 1 \ \ \text{for all} \ \ i,j = 1, \ldots , \ell. 
$$
Now, choose $p_{\ell+1} \in \mathbb{P}_{1(4)}$ to satisfy $p_{\ell+1} \equiv 1(\mathrm{mod}\: q_1 \cdots q_{\ell})$, and
$q_{\ell+1} \in \mathbb{P}_{1(4)}$ to satisfy $q_{\ell+1} \equiv 1(\mathrm{mod}\: p_1 \cdots p_{\ell+1})$. Then  
$$
\left( \frac{p_{\ell+1}}{q_j} \right) = \left( \frac{1}{q_j}  \right) = 1 \ \ \text{for} \ \ j = 1, \ldots , \ell, \ \ \text{and} \ \ \left( \frac{p_{i}}{q_{\ell+1}} \right) = 
\left( \frac{q_{\ell+1}}{p_{i}} \right) = \left( \frac{1}{p_{i}} \right) = 1 \ \ \text{for} \ \ i = 1, \ldots , \ell+1
$$
by quadratic reciprocity, as required. (It follows from our construction that $p_{\ell} > p_1^{\ell-1}$ for all $\ell > 1$, which easily implies that the set of primes $P$ has Dirichlet density zero. In fact, it follows from Proposition \ref{P:1000} that there is no example of this kind with $P$ having positive density.) } 
\end{example}

%In connection with the last remark in the previous example, it would be interesting to exhibit a set of primes $P$ of positive Dirichlet density such that for $K = \mathbb{Q}$ and $S = \{v_{\infty}\} \cup \{v_p \, \vert \, p \in P \}$, the index $[\mathbb{I}(S) : \overline{\QQ^{\times}}^{(S)}]$ is infinite (see below regarding nonsplit tori). We, however, turn now to an example of a {\it special} set of primes having positive Dirichlet density (that comes from an arithmetic progression) for which $T = \mathbb{G}_m$ does have strong approximation. 

We now turn to our first example of strong approximation with respect to a {\it special} set of primes. 

\begin{example}\label{extori3}
\rm{Let $S = \{v_{\infty}\} \cup \{ v_p \, \vert \, p \in \mathbb{P}_{1(4)} \}$. We will now show $T = \mathbb{G}_m$ over $K = \mathbb{Q}$ has strong approximation with respect to $S$. As we noted above, due to $h(\mathbb{Q}) = 1$, it is enough to show that the subgroup $\mathbb{E}_S$, which is generated by $-1$ and all primes $p \in \mathbb{P}_{1(4)}$, is dense in $\mathbb{U}_S$. Since sets of the form  (\ref{E:basis}) constitute a basis of open sets for the $S$-idelic topology, it is enough to show that every set of the form 
$$
U \, = \, \prod_{i = 1}^r \left(a_i + p_i^{\alpha_i} \mathbb{Z}_{p_i} \right) \, \times \, \prod_{q \in \mathbb{P} \setminus (\mathbb{P}_{1(4)} \cup P)} \mathbb{Z}_q^{\times} 
$$
where $P = \{p_1 = 2, p_2, \ldots , p_r\} \subset \mathbb{P} \setminus \mathbb{P}_{1(4)}$ is a finite subset, and $\alpha_i \geq 1$ and $a_i$ with $(a_i , p_i) = 1$ are integers for $i = 1, \ldots , r$, intersects $\mathbb{E}_S$. Set $\varepsilon = 1$ if $a_1 \equiv 1(\mathrm{mod}\: 4)$, and $\varepsilon = -1$ if $a_1 \equiv 3(\mathrm{mod}\: 4)$, so that $\varepsilon a_1 \equiv 1(\mathrm{mod}\: 4)$ in all cases. Using the Chinese Remainder Theorem, we can find $c \in \mathbb{Z}$ satisfying 
$$
\begin{cases}
c \equiv \varepsilon a_i(\mathrm{mod}\: p_i^{\alpha_i}) \  \text{for} \ i = 1, \ldots , r, \\ 
c \equiv 1(\mathrm{mod}\: 4).
\end{cases}
$$
Next, by Dirichlet's Theorem on arithmetic progressions, there exists a prime $p \equiv c(\mathrm{mod}\: 4p_1^{\alpha_1} \cdots p_r^{\alpha_r})$. Then $\varepsilon p \in \mathbb{E}_S \cap U$. This completes the proof of the fact that $\overline{\mathbb{E}_S}^{(S)} = \mathbb{U}_S$.}
\end{example}

It follows from our general results that for any integers $a , m$ with $(a , m) = 1$ and $$S=\{v_\infty\}\cup\{v_p\,|\,p\in\mathbb{P}_{a(m)}\},$$ the index $[\mathbb{I}_S : \overline{\QQ^{\times}}^{(S)}]$ is always finite (cf. Proposition \ref{P:1}). However, as we will now show, this index is not always equal to one. 

\begin{example}\label{exampleASAindex}
\rm{Let $K = \mathbb{Q}$, let $q$ be a prime $\equiv 1(\mathrm{mod}\: 4)$, and set $S = \{v_{\infty}\} \cup \{ v_p \, \vert \, p \in \mathbb{P}_{1(q)}\}$. Then 
\begin{equation}\label{E:gr1}
[\mathbb{I}_S : \overline{\QQ^{\times}}^{(S)}] > 1. 
\end{equation}
Our proof will use the Artin map associated with the quadratic extension $L = \QQ(\sqrt{q})$, and we would like to point out that a suitable generalization of this approach will play a crucial role also in the proof of Proposition \ref{P:1}. We let $(* , *)_p$ (resp., $(* , *)_{\infty}$) denote the Hilbert symbol over $\mathbb{Q}_p$ (resp., over $\mathbb{R}$). If we identify the Galois group $\mathrm{Gal}(L/\QQ)$ with $\{ \pm 1 \}$, then the Artin map $\psi_{L/\QQ} \colon \mathbb{I} \to \mathrm{Gal}(L/\QQ)$ is given by 
$$
(x_p)_p \mapsto (x_{\infty} , q)_{\infty} \cdot \prod_{p \in \mathbb{P}} (x_p , q)_p.
$$
Then by class field theory for $L/\QQ$, the kernel $N := \ker \psi_{L/\QQ} \subset \mathbb{I}$ is an open subgroup containing $\QQ^{\times}$ and having index two. Let $\pi_S \colon \mathbb{I} \to \mathbb{I}_S$ be the canonical projection. Then $\pi_S(N)$ is an open subgroup of $\mathbb{I}_S$ containing $\QQ^{\times}$, and to prove (\ref{E:gr1}) it is enough to show that $N \supset \ker \pi_S$ as then $\pi_S(N) \neq \mathbb{I}_S$. For this we observe that $q \in {\QQ_v^{\times}}^2$ for all $v \in S$. This is obvious for $v=v_\infty$, and follows from 
$$
\left( \frac{q}{p}\right) = \left( \frac{p}{q} \right) = \left( \frac{1}{q}\right) = 1
$$
for $v=v_p$ with $p \in \mathbb{P}_{1(q)}$. Let now $x = (x_p)_p \in \ker \pi_S$, i.e. $x_p = 1$ for $p \in \mathbb{P} \setminus \mathbb{P}_{1(q)}$. Then 
$$
\psi_{L/\QQ}(x) = (x_{\infty} , q)_{\infty} \cdot \prod_{p \in \mathbb{P}_{1(q)}} (x_p , q)_p = 1, 
$$
proving that $N \supset \ker \pi_S$. (Using the cyclotomic extension $\QQ(\zeta_q)$ in place of $L$ in the above argument, one can show that the index in (\ref{E:gr1}) can be arbitrarily large.)
}
\end{example}

The main takeaway from this example is that the property that can be expected to hold for tori with respect to arithmetic progressions is {\it not} strong approximation in the classical sense but rather the following variation (in fact, a slight weakening) of the latter which we define for \underline{arbitrary} algebraic groups (and not just for tori). 

\begin{definition}\label{D:ASA}
Let $G$ be a linear algebraic group defined over a global field $K$. We say that $G$ has {\it almost strong approximation} with respect to a subset $S \subset V^K$ if the closure $\overline{G(K)}^{(S)}$ of the group of $K$-rational points $G(K)$ diagonally embedded in the group of $S$-adeles $G(\mathbb{A}_{K,S})$ has finite index. 
\end{definition}

Thus, we can reformulate the result mentioned prior to Example \ref{exampleASAindex} by saying that $T = \mathbb{G}_m$ over $K = \mathbb{Q}$ has almost strong approximation for any set of the form $S = \{v_{\infty}\} \cup \{ v_p \, \vert \, p \in \mathbb{P}_{a(m)}\}$ with $(a , m) = 1$.

We also observe that if $[G(\mathbb{A}_{K,S}) : \overline{G(K)}^{(S)}] < \infty$ then there exists a finite subset $W \subset V^K \setminus S$ such that $\overline{G(K)}^{(S \cup W)} = G(\mathbb{A}_{K,S \cup W})$, i.e. $G$ has strong approximation with respect to $(S \cup W)$. Indeed, we may assume that $S \supset V^K_{\infty}$. Since $\overline{G(K)}^{(S)}$ is open in $G(\mathbb{A}_{K,S})$, we can find a finite subset $W_1 \subset V^K \setminus S$ for which $\displaystyle \overline{G(K)}^{(S)} \supset \prod_{v \notin S \cup W_1} G(\mathcal{O}_v)$. Now, let $\{ g^j\}_{j = 1}^t$ be a system of right coset representatives  of $G(\mathbb{A}_{K,S})$ by $\overline{G(K)}^{(S)}$, where $g^j = (g^j_v)_v$. Then there exists a finite subset $W_2 \subset V^K \setminus S$ such that
$$
g^j_v \in G(\mathcal{O}_v) \ \ \text{for all} \ \ j= 1, \ldots , t \ \ \text{and all} \ \ v \in V^K \setminus (S \cup W_2).
$$
Set $W := W_1 \cup W_2$. Then projecting $\displaystyle G(\mathbb{A}_{K,S}) = \bigcup_{j = 1}^t g^j \overline{G(K)}^{(S)}$ to $G(\mathbb{A}_{K,S \cup W})$, we obtain $G(\mathbb{A}_{K,S \cup W}) = \overline{G(K)}^{(S \cup W)}$, as required. This observation justifies the term ``almost strong approximation.''  We note, however, that there is no constructive way to determine $W$ or even its size from  information about the index  $[G(\mathbb{A}_{K,S}) : \overline{G(K)}^{(S)}]$. 

The last point we need to make before formulating the general result is that almost strong approximation with respect to arithmetic progressions in {\it nonsplit} tori encounters additional obstructions of arithmetic nature. 

\begin{example}\label{noASAbctechcond}
Let $K = \mathbb{Q}$ and let $S = \{v_{\infty}\} \cup \{ v_p \, \vert \, p \in \mathbb{P}_{3(4)}\}$. We consider the norm torus $T = \mathrm{R}_{L/\QQ}^{(1)}(\mathbb{G}_m)$ (cf. \cite[Ch. 2, \S2.1.7]{PlRa2}) associated with the extension $L = \QQ(i)$, $i^2 = -1$. Then 
\begin{equation}\label{E:infty10}
[T(\mathbb{A}_S) : \overline{T(\QQ)}^{(S)}] = \infty.
\end{equation}
To prove this, we adapt the strategy used above, viz. we introduce the open subgroup 
$$
\mathbb{U}_{T,S} := \prod_{p \in \mathbb{P} \setminus \mathbb{P}_{3(4)}} T(\mathbb{Z}_p),  
$$
where we consider the standard integral structure on $T$ for which $T(\mathbb{Z}_p)$ is the maximal compact subgroup of $T(\mathbb{Q}_p)$,  and set $\mathbb{E}_{T,S} := T(\QQ)\cap \mathbb{U}_{T,S}$. It is enough to show that $\mathbb{E}_{T,S}$ is finite, as then $[\mathbb{U}_{T,S} : \overline{\mathbb{E}_{T,S}}^{(S)}] = \infty$, and (\ref{E:infty10}) will follow. 

For this, following the referee's suggestion, we observe that $T$ is the variety defined by the equation $x^2 + y^2 - 1 = 0$, and it is easy to check that $T(\mathbb{Q}_p) = T(\mathbb{Z}_p)$ for all $p\in \mathbb{P}_{3(4)}$. So, 
$$
\mathbb{E}_{T, S} = T(\QQ) \bigcap \left( \prod_{p \in \mathbb{P}_{3(4)}} T(\mathbb{Q}_p) \times \prod_{p \in \mathbb{P} \setminus \mathbb{P}_{3(4)}}T(\ZZ_p)  \right) = T(\QQ) \bigcap \prod_{p\in\mathbb{P}} T(\mathbb{Z}_p) = T(\mathbb{Z}), 
$$
which coincides with $\mathbb{E} := \{\pm 1, \pm i \}$. Since $\mathbb{E}$ is finite, $\mathbb{E}_{T, S}$ is also finite, as required.

%Now, let $x \in \mathbb{E}_{T,S}$. By Hilbert's Theorem 90, we can write $x = \sigma(y)/y$ for %some $y \in L^{\times}$, where $\sigma \in \mathrm{Gal}(L/\QQ)$ is the nontrivial automorphism. %For $p \notin \mathbb{P}_{1(4)}$, the valuation $v_p$ has a unique extension $w_p$ to $L$, and we %set $\alpha_p: = w_p(y)$. For $p \in \mathbb{P}_{1(4)}$, however, $v_p$ has two extensions $w'_p %, w''_p$ (swapped by $\sigma$), but since $x \in T(\mathbb{Z}_p)$ we have $w'_p(y) = w''_p(y) =: %\alpha_p$. Set 
%$$
%a := \prod_{p\neq 2} p^{\alpha_p} \ \ \text{and} \ \ z := y/a.
%$$
%Then $x = \sigma(z)/z$ and by construction $w(z) = 0$ for all $w \in V^L_f \setminus \{ w_2 \}$. %Since the prime element of $\mathbb{Z}[i]$ lying above $2$ is $(1+i)$, we conclude that $z = %\varepsilon \cdot (1+i)^{\ell}$ with $\varepsilon \in \mathbb{E} := \{ \pm 1, \pm i \}$ (the unit %group of $\mathbb{Z}[i]$) and $\ell \in \mathbb{Z}$. Then $x = \sigma(z)/z \in \mathbb{E}$, so  %$\mathbb{E}_{T,S} = \mathbb{E}$ is finite, as required.}
\end{example}

With regard to this example, we note that $\mathbb{P}_{3(4)}$ can be described as the generalized arithmetic progression $\mathcal{P}(L/\QQ , \{\sigma\})$, where $L = \QQ(i)$ and $\sigma \in \mathrm{Gal}(L/\QQ)$ is the nontrivial automorphism, and the splitting field of the torus $T$ is also $L$. Thus, situations of this kind create an obstruction to almost strong approximation. More generally, it was shown in \cite[Proposition 4]{PrRa-Irr} that given a generalized arithmetic progression $\mathcal{P}(L/K , \{\sigma\})$, where $L/K$ is a finite abelian extension of an arbitrary global field $K$ and $\sigma \in \mathrm{Gal}(L/K)$, and a finite separable extension $F/K$ such that $\sigma \vert (F \cap L) \neq \mathrm{id}_{F \cap L}$, for the torus $T = \mathrm{R}_{F/K}(\mathbb{G}_m)$ the quotient $T(\mathbb{A}_{K,S}) / \overline{T(K)}^{(S)}$ has infinite exponent. The theorem below shows that tori do have almost strong approximation with respect to tractable sets in the absence of obstructions of this kind. In fact, it applies to more general sets that can be characterized in terms of Dirichlet density. To relate these two types of sets, we recall that by Dirichlet Theorem on Primes in Arithmetic Progressions, for any relatively prime integers $a$ and $m$ the set $\mathbb{P}_{a(m)}$ has Dirichlet density $1/\varphi(m)$ (where $\varphi$ is Euler's totient function), and more generally by Chebotarev Density Theorem (cf. \cite[ch. VII, \S 2]{ANT}), $\mathcal{P}(L/K , \mathcal{C})$ has Dirichlet density $\vert \mathcal{C} \vert / [L : K]$.

%Our main result states that in the absence of obstruction of this kind, a torus does have almost strong approximation with respect to any tractable set of valuations. 

\begin{thm}\label{T:1}
Let $K$ be a global field, and let $T$ be a $K$-torus with the minimal splitting field $P/K$. If a subset $S \subset V^K$ contains $V^K_{\infty}$ and satisfies $\mathfrak{d}_K(S \cap \mathrm{Spl}(P/K)) > 0$ then $T$ has almost strong approximation with respect to $S$. 

Furthermore, if $S$ is a tractable set of valuations containing a set of the form $V^K_{\infty} \cup (\mathcal{P}(L/K , \mathcal{C}) \setminus \mathcal{P}_0)$, where $\mathcal{P}(L/K , \mathcal{C})$ is a generalized arithmetic progression associated with a finite Galois extension $L/K$ and a conjugacy class $\mathcal{C} \subset \mathrm{Gal}(L/K)$, and $\mathcal{P}_0$ has Dirichlet density zero, such that 
\begin{equation}\label{technicalconditionarbtori}
\sigma \vert (P \cap L) = \mathrm{id}_{P \cap L} \ \ \text{for some} \ \ \sigma \in \mathcal{C}, 
\end{equation}
then the index $[T(\mathbb{A}_{K,S}) : \overline{T(K)}^{(S)}]$ divides an explicit constant $\tilde{C}(d , n)$ that depends only on $d = \dim T$ and $n = [L : K]$.  
\end{thm}

As we will see in the next section after Proposition \ref{P:1000}, tractable sets $S$ as in the statement do satisfy the condition $\mathfrak{d}_K(S \cap \mathrm{Spl}(P/K)) > 0$.

\section{Proof of Theorem \ref{T:1}:  quasi-split tori}\label{ASAquasisplit}

We recall that a $K$-torus $T$ is called {\it quasi-split} if it is isomorphic to a finite product of $K$-tori of the form $\mathrm{R}_{F/K}(\mathbb{G}_m)$ where $F/K$ is a finite separable extension. The goal of this section is to prove 
Theorem \ref{T:1} for such tori with a very explicit constant in the case of tractable sets -- see Theorem \ref{T:2}. 

First, we treat the case of $T = \mathrm{R}_{F/K}(\mathbb{G}_m)$. For a subset $S \subset V^K$ we let $\bar{S} \subset V^F$ denote the set of all extensions of valuations $v \in S$ to $F$. We recall that for any $K$-algebra $B$ there is a natural isomorphism 
$$
T(B) \simeq (B \otimes_K F)^{\times}
$$
(cf. \cite[A.5]{CGP}). Combining this with the natural isomorphism $\mathbb{A}_K \otimes_K F \simeq \mathbb{A}_F$ (cf. \cite[Ch. 2, \S14]{ANT}), which induces an isomorphism $\mathbb{A}_{K,S} \otimes_K F \simeq \mathbb{A}_{F,\bar{S}}$, we obtain 
compatible isomorphisms
$$ 
T(K) \simeq F^{\times} \ \ \text{and} \ \ T(\mathbb{A}_{K,S}) \simeq \mathbb{I}_{F,\bar{S}},
$$
leading to an isomorphism 
$$
T(\mathbb{A}_{K,S}) / \overline{T(K)}^{(S)} \ \simeq \ \mathbb{I}_{F,\bar{S}} / \overline{F^{\times}}^{(\bar{S})}, 
$$
where $\overline{F^{\times}}^{(\bar{S})}$ is the closure of (diagonally embedded) $F^{\times}$ in $\mathbb{I}_{F,\bar{S}}$. The index $[\mathbb{I}_{F,\bar{S}} : \overline{F^{\times}}^{(\bar{S})}]$ is analyzed in Propositions \ref{P:1}-\ref{P:2000}, the first of which generalizes \cite[Proposition 5.1]{PrRa} which treats the case where $\mathcal{P}_0$ is finite.  
\begin{prop}\label{P:1}
Let $F$ be a finite separable extension of a global field $K$, and let $S \subset V^K$ be a tractable subset containing a set of the form $V^K_{\infty} \cup (\mathcal{P}(L/K , \mathcal{C}) \setminus \mathcal{P}_0)$ where $\mathcal{P}_0$ has Dirichlet density zero. Assume that there exists $\sigma \in \mathcal{C}$ such that 
\begin{equation}\label{technicalcondition}
\sigma \vert {(F \cap L)}=\mathrm{id}_{F\cap L}. 
\end{equation}
Then the index $[\mathbb{I}_{F,\bar{S}} :\overline{F^{\times}}^{(\bar{S})}]$ is finite and divides $n = [L:K]$.
\end{prop}

\begin{proof}
It is observed in the proof of Proposition 5.1 in \cite{PrRa} that the quotient $\mathbb{I}_{F,\bar{S}}/\overline{F^\times}^{(\bar{S})}$ is a profinite group, and consequently 
$$
\overline{F^{\times}}^{(\bar{S})} = \bigcap_{B\in\mathcal{B}} B,
$$
where $\mathcal{B}$ denotes the family of all open subgroups of $\mathbb{I}_{F,\bar{S}}$ that contain $F^{\times}$. Every such subgroup $B$ is automatically of finite index,
and it is enough to show that

\vskip1.5mm
\begin{equation*}
\text{for every $B\in\mathcal{B}$, the index }[\mathbb{I}_{F,\bar{S}} : B] \text{ divides } m = [FL :F]\tag{$\star$}
\end{equation*}
\vskip1.5mm

\noindent 

\noindent as obviously $m \vert n$. Let $\pi_{\bar{S}} \colon \mathbb{I}_F \to \mathbb{I}_{F,\bar{S}}$ be the natural projection, and set $M := \pi_{\bar{S}}^{-1}(B)$. Then $M$ is an open subgroup of $\mathbb{I}_F$ containing $F^{\times}$ and having index $[\mathbb{I}_F : M] = [\mathbb{I}_{F,\bar{S}} : B]$. By the  Existence Theorem of global class field theory (cf. \cite[Ch. VII, Theorem 5.1(D)]{ANT}) there is an abelian extension $P/F$ such that $M= N_{P/F}(\mathbb{I}_P) F^{\times}$. On the other hand, due to the fundamental isomorphism of global class field theory (see \cite[p. 197]{ANT}), the index $[\mathbb{I}_F : N]$ of the norm subgroup $N :=  N_{FL/F}(\mathbb{I}_{FL}) F^{\times}$ divides $[FL : F] = m$. 
Thus, it is enough to prove the inclusion $P \subset FL$ as then $N \subset M$. 

Suppose that $P \not\subset FL$. Pick $\sigma \in \mathcal{C}$ that satisfies (\ref{technicalcondition}), and using the canonical isomorphism of Galois groups
$\mathrm{Gal}(FL/F) \simeq \mathrm{Gal}(L/(F \cap L))$, find $\tilde{\sigma} \in \mathrm{Gal}(FL/F)$ such that $\tilde{\sigma} \vert L = \sigma$. Let $E$ be a finite Galois extension of $K$ that contains $F$, $L$ and $P$.  We claim that there exists $\tau \in \mathrm{Gal}(E/K)$ such that
\begin{equation}\label{E:restr}
\tau \vert FL = \tilde{\sigma} \ \ \text{and} \ \ \tau \vert P \neq \mathrm{id}_P.
\end{equation}
Indeed, otherwise the set of all $\tau \in \mathrm{Gal}(E/K)$ satisfying $\tau \vert FL = \tilde{\sigma}$, which is a right coset of the subgroup $\mathrm{Gal}(E/FL)$, would be contained in $\mathrm{Gal}(E/P)$. This would imply the inclusion $\mathrm{Gal}(E/FL) \subset \mathrm{Gal}(E/P)$ yielding the inclusion $P \subset FL$ that contradicts our original assumption.  

So, fix $\tau \in \mathrm{Gal}(E/K)$ satisfying (\ref{E:restr}). By Chebotarev's Density Theorem (cf. \cite[Ch. VII, \S 2]{ANT}), the set of $v \in V^K_f$ that are unramified in $E$ and admit an extension $w \in V^E_f$ such that $\mathrm{Fr}_{E/K}(w | v) = \tau$ has positive Dirichlet density. Since $\mathfrak{d}_K(\mathcal{P}_0) = 0$ by our assumption, such a $v$ can actually be found outside of $\mathcal{P}_0$; we then let $w$ denote an extension of $v$ as above. 
The fact that $\tau \vert L = \sigma$ implies that $v \in S$, placing $u' := w \vert F$ in $\bar{S}$. Furthermore, since $\tau$ generates $\mathrm{Gal}(E_w/K_v)$, the facts that $\tau \vert F = \mathrm{id}_F$ and $\tau \vert P \neq \mathrm{id}_P$ mean that $F_{u'} = K_v$ while for $u'' := w \vert P$ we have $P_{u''} \neq F_{u'}$ (note that $u'' \vert u'$). On the other hand, since $u' \in \bar{S}$, it follows from the construction of $M$ that we have the inclusion $F_{u'}^{\times} \subset M$. But $M$ coincides with the kernel of the Artin map $\psi_{P/F} \colon \mathbb{I}_F \to \mathrm{Gal}(P/F)$, so the restriction of $\psi_{P/F}$ to $F_{u'}^{\times} \subset \mathbb{I}_F$ is trivial. Since by construction of the Artin map we have 
$$\psi_{P/F}(F_{u'}^{\times}) = \langle \mathrm{Fr}_{P/F}( u'' \vert u') \rangle = \mathrm{Gal}(P_{u''}/F_{u'}),$$ 
we obtain that $P_{u''} = F_{u'}$, a contradiction. Thus, $P \subset FL$, completing the argument. 
\end{proof}

As we already mentioned in the introduction, the preprint \cite{CW} by Y.~Cao and Y.~Wang has led 
to the following modification of Proposition \ref{P:1}. 
%For any finite extension $F/K$, the Dirichlet density $\mathfrak{d}_F$ of subsets of $V^F_f$ is defined similarly. We also let $\mathrm{Spl}(F/K)$ denote the set of all $v \in V^K_f$ that split completely in $F$ (cf. \cite[p. 49]{Neu}). 
\begin{prop}\label{P:1000}
Let $F/K$ be a finite (separable) extension, $S \subset V^K$ be a subset containing $V^K_{\infty}$, and assume that there exists a subset $S_0 \subset S \cap \mathrm{Spl}(F/K)$ with $\mathfrak{d}_K(S_0) > 0$. Then the index 
$$
[\mathbb{I}_{F,\bar{S}} : \overline{F^{\times}}^{(\bar{S})}] \leq \frac{1}{[F : K] \cdot \mathfrak{d}_K(S_0)} := d, 
$$ 
hence in particular is finite. 
\end{prop}
\begin{proof}
As in the proof of Proposition \ref{P:1}, it is enough to show that for any open subgroup $B$ of $\mathbb{I}_{F,\bar{S}}$ containing $F^{\times}$, the index $[\mathbb{I}_{F,\bar{S}} : B]$, which is automatically finite, is $\leq d$. Again, we let $\pi_{\bar{S}} \colon \mathbb{I}_F \to \mathbb{I}_{F,\bar{S}}$ denote the natural projection, introduce the open subgroup  $M := \pi_{\bar{S}}^{-1}(B)$ of $\mathbb{I}_F$, for which we have $[\mathbb{I}_F : M] = [\mathbb{I}_{F,\bar{S}} : B]$, and then consider the finite abelian extension $P/F$, provided by the Existence Theorem, satisfying $N_{P/F}(\mathbb{I}_P)F^{\times} = M$ (of course, $[P : F] = [\mathbb{I}_F : M]$). By construction, $F_w^{\times} \subset M$ for all $w \in  \bar{S}$.    As in the proof of Proposition \ref{P:1}, the construction of Artin's map immediately gives the inclusion 
$\bar{S} \setminus (V^F_{\infty} \cup S_r) \subset \mathrm{Spl}(P/F)$ where $S_r \subset V^F_f$ is the finite set of valuations that ramify in $P$. In particular, for the set $\bar{S}_0$ of all extensions of valuations from $S_0$ to $F$ we have 
\begin{equation}\label{E:XXXX31}
\bar{S}_0 \setminus S_r \subset \mathrm{Spl}(P/F). 
\end{equation}
It is easy to see  that 
%(cf. \cite{Milne}, Ch. VI, Prop. 3.2 and Cor. 4.6) that 
\begin{equation}\label{E:XXXX2}
\mathfrak{d}_F(\bar{S}_0) = [F : K] \cdot \mathfrak{d}_K(S_0). 
\end{equation}
For the reader's convenience, we will sketch the argument. First, we recall  that the  Dirichlet density $\mathfrak{d}_K(\mathcal{P})$ of a subset $\mathcal{P} \subset V^K_f$ is defined to be the limit 
$$
\lim_{s \to 1+} \frac{\displaystyle \sum_{v \in \mathcal{P}} \frac{1}{(\mathrm{N}v)^s}}{\displaystyle \sum_{v \in V^K_f} \frac{1}{(\mathrm{N}v)^s}}, 
$$ 
where $\mathrm{N}v$ denotes the cardinality of the residue field of $v$, if this limit exists (cf. \cite[Ch. VII, \S 13]{Neu}). Dirichlet density $\mathfrak{d}_F$ is defined similarly.

%We will freely identify valuations $v \in V^K_f$ (resp., $w \in V^F_f$) with the corresponding prime ideals $\mathfrak{p}$ of the ring of integers $\mathcal{O}_K$ (resp., prime ideals $\mathfrak{P}$ of the ring of %integers $\mathcal{O}_F$). 
%For $v \in V^K_f$ and $w \in V^F_f$ we let $\mathrm{N}v$ and $\mathrm{N}w$ denote the number of elements of the corresponding residue field. Then for a subset $R \subset V^K_f$ the Dirichlet density %$\mathfrak{d}_K(R)$ is defined as the limit
%$$
%\lim_{s \to 0+} \frac{\displaystyle \sum_{v \in R} \frac{1}{(\mathrm{N}v)^s}}{\displaystyle \sum_{v \in V^K_f} \frac{1}{(\mathrm{N}v)^s}}
%$$ 
%if this limit exists. The Dirichlet density $\mathfrak{d}_F(\bar{R})$ of a subset $\bar{R} \subset V^F_f$ is defined similarly. 

Each $v \in S_0$ has precisely $[F : K]$ extensions $w \in \bar{S}_0$, and for each extension we have $F_w = K_v$, hence $\mathrm{N}w = \mathrm{N}v$.
%Now, set $T = \{ w \in V^F_f \ \vert \ F_w = K_v \ \ \text{where} \ \ w \vert v \}$. Then $\bar{S} \cap T$ is precisely the set of all extensions of valuations $v \in S_F : = S \cap \mathrm{Spl}(F/K)$. Furthermore, each $v \in S_F$ has $[F : K]$ extensions $w$, and for any of them we have $\mathrm{N}w = \mathrm{N}v$. 
It follows that 
\begin{equation}\label{E:XXXX3}
\sum_{w \in \bar{S}_0} \frac{1}{(\mathrm{N}w)^s}
=
[F : K] \cdot
\sum_{v \in S_0} \frac{1}{(\mathrm{N} v)^s}.
\end{equation}
On the other hand, a standard argument using Euler products shows that for the zeta-functions $\zeta_K(s)$ and $\zeta_F(s)$ of $K$ and $F$ respectively, the quantities 
$$
\log \zeta_K(s) - \sum_{v \in V^K_f} \frac{1}{(\mathrm{N}v)^s} \ \ \text{and} \ \ \log \zeta_F(s) - \sum_{w \in V^F_f} \frac{1}{(\mathrm{N}w)^s} 
$$
remain bounded as $s \to 1+$ (cf. \cite[Ch. VII, \S 13]{Neu}). Since both $\zeta_K(s)$ and $\zeta_F(s)$ have a simple pole at $s = 1$ (cf. \cite[Ch. VII, Cor. 5.11]{Neu}), the difference $\log \zeta_F(s) - \log \zeta_K(s)$ remains bounded in a neighborhood of $s = 1$, and therefore the ratio of $\sum_{w \in V^F_f} \frac{1}{(\mathrm{N}w)^s}$ to $\sum_{v \in V^K_f} \frac{1}{(\mathrm{N}v)^s}$ tends to $1$ as $s \to 1+$, and (\ref{E:XXXX2}) follows from the definition of $\mathfrak{d}_F$.

On the other hand, since $\mathfrak{d}_F(\mathrm{Spl}(P/F)) = 1/[P: F]$ by Chebotarev Density Theorem, we obtain from (\ref{E:XXXX31}) that 
$$
\mathfrak{d}_F(\bar{S}_0)  \leq \frac{1}{[P : F]} = \frac{1}{[\mathbb{I}_{F,\bar{S}} : B]},    
$$
and the required fact follows. 

%Then (\ref{E:XXXX3}) and the definition of density give that 
%$$
%\mathfrak{d}_F(\bar{S} \cap T) = [F : K] \cdot \mathfrak{d}_K(S_F). 
%$$
%Since $\mathfrak{d}_F(\bar{S}) = \mathfrak{d}_F(\bar{S} \cap T)$ (cf. \cite[4.7.1]{Jan}, %\cite[Cor. 4.6]{Milne}), this yields (\ref{E:XXXX2}). 
%
%Combining (\ref{E:XXXX1}) and (\ref{E:XXXX2}), we obtain our claim. 
\end{proof}

%Since $\mathfrak{d}_F(\mathrm{Spl}(P/F)) = 1/[P : F]$ by Chebotarev Density Theorem, we obtain the inequality 
%\begin{equation}\label{E:XXXX1}
%\mathfrak{d}_F(\bar{S} \setminus V^F_{\infty}) \leq \frac{1}{[P : F]} = \frac{1}{[\mathbb{I}_{F}%(\bar{S}) : B]}.   
%\end{equation} 

Of course, Proposition \ref{P:1000} allows for sets $S$ that are more general than arithmetic progressions, but let us now see how Propositions \ref{P:1} and \ref{P:1000} relate for generalized arithmetic progressions. So, let $S = V^K_{\infty} \cup \mathcal{P}(L/K , \mathcal{C})$, and assume first that $F/K$ is a finite {\it Galois extension}. In this case, if {\it one} $\sigma \in \mathcal{C}$ satisfies (\ref{technicalcondition}) then {\it all} $\sigma \in \mathcal{C}$ satisfy this condition. Furthermore, (\ref{technicalcondition}) implies that an arbitrary $\sigma \in \mathcal{C}$ extends uniquely to $\tilde{\sigma} \in \mathrm{Gal}(FL/F) \subset \mathrm{Gal}(FL/K)$, and we let $\tilde{\mathcal{C}}$ denote the conjugacy class of $\tilde{\sigma}$ in $\mathrm{Gal}(FL/K)$ (note that this conjugacy class does not depend on the choice of $\sigma$, and in fact the restriction $\mathrm{Gal}(FL/F) \to \mathrm{Gal}(L/K)$ yields a bijection $\tilde{\mathcal{C}} \to \mathcal{C}$). It is easy to see that 
$$
\mathcal{P}(L/K , \mathcal{C}) \cap \mathrm{Spl}(F/K) = \mathcal{P}(FL/K , \tilde{\mathcal{C}}), 
$$   
so applying Chebotarev's Theorem we obtain that 
$$
\mathfrak{d}_K(\mathcal{P}(FL/K , \tilde{\mathcal{C}})) = \frac{\vert \tilde{\mathcal{C}} \vert}{[FL : K]} = \frac{\vert \mathcal{C} \vert}{[FL : K]}.  
$$ 
Now Proposition \ref{P:1000} yields the following estimate: 
$$
\mathbf{i} := [\mathbb{I}_{F,\bar{S}} : \overline{F^{\times}}^{(\bar{S})}] \leq \left( \frac{\vert \mathcal{C} \vert}{[FL : K]} \cdot [F : K]   \right)^{-1} = \frac{[FL : F]}{\vert \mathcal{C} \vert}.  
$$
At the same time, the proof of Proposition \ref{P:1} shows that $\mathbf{i}$ divides $[FL : F]$. So, in the Galois case Proposition \ref{P:1000} subsumes Proposition \ref{P:1} and provides a better estimate of the index of the closure, although this estimate does not have to be an integer let alone a multiple of the index (as in Proposition \ref{P:1}). 

On the other hand, let us consider a non-Galois extension $F = \mathbb{Q}(\sqrt[3]{2})$ of $K = \mathbb{Q}$, and take $S = V^K_{\infty} \cup \mathcal{P}(L/K , \mathcal{C})$ where $L = \mathbb{Q}(\zeta_3)$ and $\mathcal{C} = \{ \sigma \}$ with $\sigma$ being the nontrivial automorphism of $L/K$. The normal closure of $F$ over $K$ is $E = \mathbb{Q}(\zeta_3 , \sqrt[3]{2})$, which contains $L$. Then $\mathrm{Spl}(F/K) = \mathrm{Spl}(E/K)$ is disjoint from $S$, and therefore Proposition \ref{P:1000} does not apply. However, Proposition \ref{P:1} does apply yielding that $[\mathbb{I}_{F,\bar{S}} : \overline{F^{\times}}^{(\bar{S})}]$ divides $2$. Thus, in the non-Galois case, Proposition \ref{P:1000} does not subsume Proposition \ref{P:1}.   

\vskip2mm 

So, while Proposition \ref{P:1000} is sufficient to prove our main results, we would like to point out a generalization of Proposition \ref{P:1000} which subsumes Proposition \ref{P:1} in all cases, 
at least qualitatively. 
For a finite (separable) extension $F/K$, we let $\mathrm{Spl}_0(F/K)$ denote the set of $v \in V^K_f$ that admit {\it at least one} extension $w$ with $F_w = K_v$, noting that $\mathrm{Spl}_0(F/K) \supset \mathrm{Spl}(F/K)$, and the two sets are equal if $F/K$ is a Galois extension.  
\begin{prop}\label{P:2000}
Let $F/K$ be a finite (separable) extension, and let $S \subset V^K$ be a set of valuations containing $V^K_{\infty}$ and such that there exists a subset $S_0 \subset S \cap \mathrm{Spl}_0(F/K)$ with 
$\mathfrak{d}_K(S_0) > 0$. Then the index 
$$
[\mathbb{I}_{F,\bar{S}} : \overline{F^{\times}}^{(\bar{S})}] \leq \frac{1}{\mathfrak{d}_K(S_0)} =: \tilde{d}, 
$$
and in particular is finite. 
\end{prop}
\begin{proof}
For each $v \in S_0$ we pick {\it exactly one} extension $w \in \bar{S}$ satisfying $F_w = K_v$, and let $\tilde{S}_0$ denote the set of such extensions for all $v \in S_0$. Then %repeating the argument we sketched in the proof of Proposition \ref{P:1000}, one shows that 
\begin{equation}\label{E:XXXX5}
\mathfrak{d}_F(\tilde{S}_0) = \mathfrak{d}_K(S_0)  
\end{equation}
as
$$
\sum_{w \in \tilde{S}_0} \frac{1}{(\mathrm{N}w)^s} = \sum_{v \in S_0} \frac{1}{(\mathrm{N}v)^s} 
$$
(see the proof of Proposition \ref{P:1000} for details). To complete the argument, let us again show that for any open subgroup $B \subset \mathbb{I}_{F,\bar{S}}$ containing $F^{\times}$ we have $[\mathbb{I}_{F,\bar{S}} : B] \leq \tilde{d}$. Continuing with the notations from the proof of Proposition \ref{P:1000}, we set $M = \pi_{\bar{S}}^{-1}(B)$, and let $P/F$ be the abelian extension for which $N_{P/F}(\mathbb{I}_P) F^{\times} = M$. Then $F_w^{\times} \subset M$ for all $w \in \bar{S}$, and therefore 
$$
\tilde{S}_0 \setminus S_r \subset \bar{S} \setminus (V^F_{\infty} \cup S_r) \subset \mathrm{Spl}(P/F), 
$$
where $S_r$ is the finite set consisting of $w \in V^F_f$ that ramified in $P$. So, 
$$
\mathfrak{d}_F(\tilde{S}_0 \setminus S_r) \leq \frac{1}{[P : F]} = \frac{1}{[\mathbb{I}_{F,\bar{S}} : B]},  
$$
and our claim follows from (\ref{E:XXXX5}).  
\end{proof}

Finally, we observe that if (\ref{technicalcondition}) holds, then $\sigma$ lifts to $\tilde{\sigma} \in \mathrm{Gal}(FL/F)$ and eventually to $\tau \in \mathrm{Gal}(E/F)$ where $E$ is the normal closure of $FL$ over $K$. Let $\mathcal{T}$ be the conjugacy class of $\tau$ in $\mathrm{Gal}(E/K)$. Then any $v \in \mathcal{P}(E/K , \mathcal{T}) =: S_0$ lies in $\mathcal{P}(L/K , \mathcal{C}) \cap \mathrm{Spl}_0(F/K)$. So, for $S = V^K_{\infty} \cup \mathcal{P}(L/K , \mathcal{C})$ the intersection $S \cap \mathrm{Spl}_0(F/K)$ contains the set $S_0$ which has positive density. 
%While showing that $S \cap \mathrm{Spl}_0(F/K)$ has Dirichlet density, which would automatically be positive, requires a little more work, the reader can check that the finiteness of $[\mathbb{I}_F(\bar{S}) : \overline{F^{\times}}^{(\bar{S})}]$ in Proposition \ref{P:1000} is guaranteed if $S \cap \mathrm{Spl}_0(F/K)$ contains a set of positive density (just change $S$!). 
Thus, the finiteness assertion in Proposition \ref{P:1} follows from Proposition \ref{P:2000}. 

\vskip2mm 

The discussion preceding Proposition \ref{P:1} entails that the assertions of Propositions \ref{P:1} and \ref{P:1000} can be restated in terms of the torus $T = \mathrm{R}_{F/K}(\mathbb{G}_m)$ (note that the minimal splitting field $P$ of $T$ coincides with the normal closure of $F$ over $K$, and hence $\mathrm{Spl}(P/K) = \mathrm{Spl}(F/K)$). Let $S \subset V^K$ be a subset containing $V^K_{\infty}$. Then 

\vskip2mm 

 $\bullet$ \parbox[t]{16cm}{if there exists $S_0 \subset S \cap \mathrm{Spl}(P/K)$ with $\mathfrak{d}_K(S_0) > 0$ then $$[T(\mathbb{A}_{K,S}) : \overline{T(K)}^{(S)}] \leq \frac{1}{[F : K] \cdot \mathfrak{d}_K(S_0)};$$}

\vskip1mm 

 $\bullet$ \parbox[t]{16cm}{if $S$ is a tractable set containing a subset of the form $V^K_{\infty} \cup (\mathcal{P}(L/K , \mathcal{C}) \setminus \mathcal{P}_0)$ where $\mathcal{P}_0$ has Dirichlet density zero and there exists $\sigma \in \mathcal{C}$ such that $\sigma \vert (P \cap L) = \mathrm{id}_{P \cap L}$ then the index $[T(\mathbb{A}_{K,S}) : \overline{T(K)}^{(S)}]$ divides $n = [L : K]$. }

\vskip2mm 

\noindent We will now extend these facts to arbitrary quasi-split tori.

\begin{thm}\label{T:2}
Let $T = \mathrm{R}_{F_1/K}(\mathbb{G}_m) \times \ldots \times \mathrm{R}_{F_r/K}(\mathbb{G}_m)$, where $F_1, \ldots , F_r$ are finite separable extensions of a global field $K$, be a quasi-split $K$-torus having dimension $d$ and the minimal splitting field $P$, and let $S \subset V^K$. If $S$ contains $V_\infty^K$ and $\mathfrak{d}_K(S \cap \mathrm{Spl}(P/K)) > 0$, then $T$ has almost strong approximation with respect to $S$. 

Furthermore, if $S$ is a tractable subset containing a set of the form $V_\infty^K\cup (\mathcal{P}(L/K , \mc{C}) \setminus \mathcal{P}_0)$ where $\mathcal{P}_0$ has Dirichlet density zero, and there exists $\sigma \in \mathcal{C}$ such that 
\begin{equation}\label{E:AAA10}
\sigma \vert (P \cap L) = \mathrm{id}_{P \cap L}, 
\end{equation}
then the index $[T(\mathbb{A}_{K,S}) : \overline{T(K)}^{(S)}]$ divides $n^r$, hence also $n^d$ where $n = [L : K]$.
\end{thm}
\begin{proof}
Let  $T_i = \mathrm{R}_{F_i/K}(\mathbb{G}_m)$ so that $T = T_1 \times \cdots \times T_r$. Then $P$ coincides with the compositum $P_1 \cdots P_r$ where $P_i$ is the minimal splitting field of $T_i$. Furthermore, we have an isomorphism 
%\begin{equation}\label{E:AAA11}
$$
T(\mathbb{A}_{K,S}) / \overline{T(K)}^{(S)} \simeq T_1(\mathbb{A}_{K,S}) / \overline{T_1(K)}^{(S)} \times \cdots \times T_r(\mathbb{A}_{K,S}) / \overline{T_r(K)}^{(S)}. 
$$
%\end{equation}
Clearly, $S_0 := S \cap \mathrm{Spl}(P/K) \subset S \cap \mathrm{Spl}(P_i/K)$ for each $i = 1, \ldots , r$, so our earlier results imply that the index 
$$
[T(\mathbb{A}_{K,S}) : \overline{T(K)}^{(S)}] \leq \frac{\mathfrak{d}_K(S \cap \mathrm{Spl}(P/K))^{-r}}{\prod_{i=1}^r [F_i : K]} \leq \mathfrak{d}_K(S \cap \mathrm{Spl}(P/K))^{-d} 
$$
(as $r \leq d$), and in particular is finite. 

Now, if $S$ is a tractable set as in the statement of the theorem then (\ref{E:AAA10}) implies that $\sigma \vert (P_i \cap L) = \mathrm{id}_{P_i \cap L}$ for each $i$. Then $[T_i(\mathbb{A}_{K,S}) : \overline{T_i(K)}^{(S)}]$ divides $n$ for each $i$, and therefore $[T(\mathbb{A}_{K,S}) : \overline{T(K)}^{(S)}]$ divides $n^r$, hence also $n^d$.  
%For each $i = 1, \ldots , r$, the (minimal) splitting field $P_i$ of $T_i$ coincides with the normal closure of $F_i$, and then $P$ is the compositum of $P_1, \ldots , P_r$. Thus, 
%our assumption (\ref{E:AAA10}) for $\sigma \in \mathcal{C}$ implies that the assumption (\ref{technicalcondition}) in Proposition \ref{P:1} holds true for $F = F_i$ and all $i = 1, \ldots , r$. As we explained above, the proposition implies that each of the indices $[T_i(\mathbb{A}_K(S)):\overline{T_i(K)}^{(S)}]$ divides $n$. 
%In view of (\ref{E:AAA11}), the index $[T(\mathbb{A}_K(S)):\overline{T(K)}^{(S)}]$ divides $n^r$, and hence also $n^d$. as clearly $r \l  
\end{proof}

\section{Some cohomological computations}\label{BoundH1} 

The passage from quasi-split to arbitrary tori in the proof of Theorem \ref{T:1} requires some auxilliary results on Galois cohomology. More specifically, the arithmetic result that we will use directly in \S \ref{ASAtori} will be formulated as Corollary \ref{C:theta} after introducing the necessary notations. We begin, however, with an elementary estimate that applies to tori over arbitrary fields.  

Let $T$ be a torus defined over a field $K$, let $K_T$ be the minimal splitting field of $T$, and let $\mathcal{G}_T = \mathrm{Gal}(K_T/K)$ be the corresponding Galois group. Then $\mathcal{G}_T$ naturally acts on the characters of $T$ making the character group $X(T)$ into a $\mathcal{G}_T$-module (cf. \cite[\S 8]{Bo}, \cite[\S2.1.7]{PlRa2}). We will also say that an integral-valued function $\psi(d)$ defined on integers $d \geq 1$ is {\it super-increasing} if for any $1 \leq d_1 \leq d_2$ we have $\psi(d_1) \vert \psi(d_2)$.   

\begin{prop}\label{boundH^1torus}
There exists an explicit positive integral-valued super-increasing function $\psi(d)$ defined on integers $d \geq 1$ such that for any $d$-dimensional torus $T$, the group $H^1(\mathcal{G}_T , X(T))$ is finite of order dividing $\psi(d)$ -- see (\ref{E:Psi-Expl}) for an explicit expression of $\psi$.
\end{prop}

We begin with the following. 

\begin{lemma}\label{estimate} Let $G$ be a finite group of order $s$ and let $A$ be a $G$-module which can be generated by $r$ elements as an abelian group. Then $H^1(G,A)$ is finite of order dividing $s^{r(s-1)}$.
\end{lemma}
\begin{proof} 
The finiteness of $H^1(G,A)$ is well-known (cf. \cite[Ch. 10, Theorem 10.29]{R}), so we will just show how to extract an explicit estimate for the order from the standard argument. 
%but since no explicit estimate for the order has been recorded, we will redo the whole argument. 
The group $H^1(G , A)$ is a quotient of the group of 1-cocycles $Z^1(G , A)$ which in turn is a subgroup of the group $F^1(G , A)$ of functions satisfying $f(1) = 0$. Clearly, $F^1(G , A) \simeq A^{s-1}$ can be generated by $\leq r(s-1)$ elements as an abelian group, and hence the same is true for $Z^1(G , A)$. On the other hand, $s = \vert G \vert$ annihilates $H^1(G , A)$ (cf. \cite[Ch. 10, Theorem 10.26]{R}), making the latter a quotient of $Z^1(G , A) / s Z^1(G , A)$. The above estimate on the number of generators yields that $Z^1(G , A) / sZ^1(G , A)$ is finite of order dividing $\vert (\mathbb{Z}/s\mathbb{Z})^{r(s-1)} \vert = s^{r(s-1)}$, and the required fact follows. 
\end{proof}

\noindent{\it Proof of Proposition \ref{boundH^1torus}.}
We have $X(T) \simeq \mathbb{Z}^d$ as abelian groups, and hence the action of $\mathcal{G}_T$ on $X(T)$ gives rise to a representation $\mathcal{G}_T \to \mathrm{GL}_d(\mathbb{Z})$. Furthermore, the fact that $K_T$ is the {\it minimal} splitting field of $T$ implies that this representation is {\it faithful}, i.e. $\mathcal{G}_T$ is isomorphic to a subgroup of $\mathrm{GL}_d(\mathbb{Z})$. It follows from the reduction theory for arithmetic groups (cf. \cite[Theorem 4.9]{PlRa2}) that $\mathrm{GL}_d(\mathbb{Z})$ has finitely many conjugacy classes of finite subgroups, so there is an integer $\gamma(d)$ depending only on $d$ such that the order of every finite subgroup of $\mathrm{GL}_d(\ZZ)$ divides $\gamma(d)$. In fact, one can give an explicit (although nonoptimal) bound $\gamma(d)$ by using Minkowski's Lemma (cf. \cite[Ch. 4, \S4.8, Lemma 4.63]{PlRa2}), according to which for any prime $p > 2$, the kernel of the reduction map $\rho_p \colon \mathrm{GL}_d(\mathbb{Z}) \to \mathrm{GL}_d(\mathbb{Z}/p\mathbb{Z})$ (in other words, the congruence subgroup of level $p$) is torsion-free, and consequently, the order of every finite subgroup of $\mathrm{GL}_d(\mathbb{Z})$ divides $\vert \mathrm{GL}_d(\mathbb{Z}/p\mathbb{Z}) \vert$. Using $p = 3$, we see that one can take 
\begin{equation}\label{E:orderGL}
\gamma(d) := \vert \mathrm{GL}_d(\mathbb{Z}/3\mathbb{Z})\vert = \prod_{i = 0}^{d-1} (3^d - 3^i)  
\end{equation}
(obviously, this function is super-increasing). Applying Lemma \ref{estimate}, we obtain that 
\begin{equation}\label{E:Psi-Expl}
\psi(d) := \gamma(d)^{d(\gamma(d) - 1)}
\end{equation}
meets our requirements. 

\hfill\(\Box\)

\begin{cor}\label{C:order}
Let $\psi(d)$ be as in the proposition. Then for every $d$-dimensional $K$-torus $T$ and any finite Galois extension $F/K$ that splits $T$, with Galois group $G = \mathrm{Gal}(F/K)$, the order of $H^1(G , X(T))$ divides $\psi(d)$. 
\end{cor}

Indeed, since $K_T$ is the minimal splitting field of $T$, we have the inclusion $K_T \subset F$, and then $\mathcal{G}_T = G/H$ where $H := \mathrm{Gal}(F/K_T)$; in fact, $H$ is the kernel of the natural action of $G$ on $X(T)$. The inflation-restriction sequence (cf. \cite[Ch. 4, \S5]{ANT}) yields the following 
exact sequence
$$0\xrightarrow{} H^1(\mc{G}_T,X(T))\xrightarrow{\mathrm{Inf}}H^1(G,X(T))\xrightarrow{\mathrm{Res}}H^1(H,X(T)).$$
Since $X(T)$ is torsion-free, we see that $H^1(H , X(T)) = \mathrm{Hom}(H , X(T))$ is trivial, so the inflation map yields an isomorphism, and our claim follows from the proposition. 

We will now combine Proposition \ref{boundH^1torus} with the Nakayama-Tate Theorem  (cf. \cite[Theorem 6, \S 11.3]{V}) to derive an arithmetic result that we will need in the next section. So, let $T$ be a torus defined over a {\it global} field $K$, let $F/K$ be a finite Galois extension that splits $T$, and let $G = \mathrm{Gal}(F/K)$ be the corresponding Galois group. We let $K[T]$ (resp., $F[T]$) denote the coordinate ring of $T$ over $K$ (resp., over $F$); recall that $F[T]$ is naturally identified with the group algebra $F[X(T)]$ and that this identification is compatible with the Galois action. For any $F$-algebra $B$ we have 
$$
T(B) = \mathrm{Hom}_{K\text{-alg}}(K[T] , B) = \mathrm{Hom}_{F\text{-alg}}(F[T] , B) = \mathrm{Hom}(X(T) , B^{\times}).
$$
Moreover, if $G$ acts on $B$ by semi-linear transformations, then the identification $T(B) = \mathrm{Hom}(X(T) , B^{\times})$ is compatible with the action of $G$. We recall that the adele ring $\mathbb{A}_F$ is equipped with a semi-linear $G$-action via the identification $\mathbb{A}_F \simeq \mathbb{A}_K \otimes_K F$ (cf. \cite[Ch. 2, \S14]{ANT}). Then the natural map of $G$-modules $T(F) \to T(\mathbb{A}_F)$ gets identified with the map $\mathrm{Hom}(X(T) , F^{\times}) \to \mathrm{Hom}(X(T) , \mathbb{I}_F)$ induced by the diagonal embedding $F^{\times} \to \mathbb{I}_F$. 

Now, we let $C_F = \mathbb{I}_F / F^{\times}$ denote the group of idele classes and set $C_F(T) = \mathrm{Hom}(X(T) , C_F)$ with the standard $G$-action. Since $X(T)$ is a free abelian group, the exact sequence 
$$
1 \to F^{\times} \longrightarrow \mathbb{I}_F \longrightarrow C_F \to 1
$$
induces an exact sequence of $G$-modules 
$$
1 \to T(F) \longrightarrow T(\mathbb{A}_F) \longrightarrow C_F(T) \to 1. 
$$
The corresponding cohomological exact sequence contains the following fragment 
\begin{equation}\label{ES:cohom}
H^1(G , T(F)) \stackrel{\theta}{\longrightarrow} H^1(G , T(\mathbb{A}_F)) \longrightarrow H^1(G , C_F(T)).
\end{equation}

\begin{prop}\label{finitecoker}
In the exact sequence (\ref{ES:cohom}), the cokernel $\coker\theta$ is finite of order dividing $\psi(d)$ (the function from Proposition \ref{boundH^1torus}).   
\end{prop}
\begin{proof}
Due to the exact sequence (\ref{ES:cohom}), the cokernel $\coker\theta$ embeds into $H^1(G , C_F(T))$. On the other hand, the Nakayama-Tate Theorem \cite[Theorem 6, \S 11.3]{V} furnishes an isomorphism 
$$
H^1(G , C_F(T)) \simeq H^1(G , X(T)). 
$$
Our claim now follows from Proposition \ref{boundH^1torus}. 
\end{proof}

We conclude this section with one consequence of Proposition \ref{finitecoker}. Let $S \subset V^K$ be an arbitrary subset, and let $\bar{S}$ be the set of all extensions of $v \in S$ to $F$. Then $G$ acts on $\mathbb{A}_{F,\bar{S}} \simeq \mathbb{A}_{K,S} \otimes_K F$ through the second factor, making $T(\mathbb{A}_{F,\bar{S}})$ into a $G$-module, with the diagonal embedding $F \hookrightarrow \mathbb{A}_{F,\bar{S}}$ yielding a homomorphism of cohomology groups 
$$
\theta_{\bar{S}} \colon H^1(G , T(F)) \longrightarrow H^1(G,  T(\mathbb{A}_{F,\bar{S}})).
$$
Furthermore, the projection $T(\mathbb{A}_F) \to T(\mathbb{A}_{F,\bar{S}})$ defines the top arrow in the following commutative diagram
$$
\begin{tikzcd}
{H^1(G , T(\AF_F))} & {H^1(G , T(\AF_{F,\bar{S}}))} \\
{H^1(G ,T(F))} & {H^1(G,T(F))}
\arrow["\nu" , from=1-1, to=1-2]
\arrow["\theta", from=2-1, to=1-1]
\arrow["=" , from=2-1, to=2-2]
\arrow["{\theta_{\bar{S}}}"', from=2-2, to=1-2]
\end{tikzcd}
$$
Since $T(\mathbb{A}_{F,\bar{S}})$ is a direct product factor of $T(\mathbb{A}_F)$ as $G$-module, $\nu$ is surjective. So, Proposition \ref{finitecoker} yields the following. 
\begin{cor}\label{C:theta}
For any subset $S \subset V^K$, the cokernel $\coker \theta_{\bar{S}}$ is finite of order dividing $\psi(d)$. 
\end{cor}

\section{Proof of Theorem \ref{T:1}: general case} \label{ASAtori}

To complete the proof of Theorem \ref{T:1}, we will relate arbitrary tori to quasi-split ones using the following fact,  which is well known (cf. \cite[Proposition 2.3]{PlRa2}) except that the dimension estimate was never recorded. 
\begin{prop}\label{exactseqtori}
Let $T$ be a torus of dimension $d$ defined over an arbitrary field $K$, and let $P$ be the minimal splitting field of $T$. Then there is an exact sequence of $K$-tori and $K$-defined morphisms
$$
1 \to T_1 \longrightarrow T_0 \stackrel{\pi}{\longrightarrow} T \to 1, 
$$
where $T_0$ is a product of $d$ copies of $\mathrm{R}_{P/K}(\mathbb{G}_m)$ (hence quasi-split), and $\dim T_1 \leq \lambda(d)$, with $\lambda$ being an explicit increasing function on integers $d \geq 1$. 
\end{prop}
\begin{proof}
Let $\mathcal{G} = \mathrm{Gal}(P/K)$. Being a free abelian group of rank $d$, the group of co-characters $X_*(T) = \mathrm{Hom}(X(T) , \mathbb{Z})$ can be generated by $d$ elements as $\mathbb{Z}[\mathcal{G}]$-module, and therefore there exists an exact sequence of $\mathbb{Z}[\mathcal{G}]$-modules of the form 
$$
0 \to Y \longrightarrow \mathbb{Z}[\mathcal{G}]^d \longrightarrow X_*(T) \to 0.  
$$
Since sending a $K$-torus to its group of co-characters defines an equivalence between the category of $K$-tori that split over $P$ and the category of finitely generated $\mathbb{Z}[\mathcal{G}]$-modules  without $\mathbb{Z}$-torsion (cf. \cite[2.1.7]{PlRa2}), corresponding to this sequence we have an exact sequence of $K$-tori and $K$-morphisms 
$$
1 \to T_1 \longrightarrow T_0 \longrightarrow T \to 1,
$$
with $X_*(T_1) = Y$ and $X_*(T_0) = \mathbb{Z}[\mathcal{G}]^d$. Then $T_0 \simeq \mathrm{R}_{P/K}(\mathbb{G}_m)^d$, and it remains to estimate $\dim T_1$ in terms of $d$. As we have seen in the proof of Proposition  \ref{boundH^1torus}, the order of $\mathcal{G}$ divides $\gamma(d)$ given by (\ref{E:orderGL}). Then 
$$
\dim T_1 = \dim T_0 - \dim T = (\vert \mathcal{G} \vert - 1)d, 
$$
so one can take $\lambda(d) := d(\gamma(d) - 1)$. 
\end{proof}

\noindent {\it Proof of Theorem \ref{T:1}.} Let $T$ be a $d$-dimensional torus defined over a global field $K$ with the minimal splitting field $P$, let $\mathcal{G} = \mathrm{Gal}(P/K)$, and let $S \subset V^K$ be a subset containing $V^K_{\infty}$ that we will specialize later. We then let $\bar{S}$ denote the set of all extensions of valuations $v \in S$ to $P$. 

We now consider the exact sequence of $K$-tori 
$$
1 \to T_1 \longrightarrow T_0 \stackrel{\pi}{\longrightarrow} T \to 1,
$$
constructed in Proposition \ref{exactseqtori} (so, in particular, $\dim T_1 \leq \lambda(d)$). The image of the embedding of abelian groups $\pi^* \colon X(T) \to X(T_0)$ has a complement (since $\coker \pi^* \simeq X(T_1)$ is torsion-free), which gives rise to a $P$-defined section $T \to T_0$ for $\pi$. It follows that for any $P$-algebra $B$, the group homomorphism $\pi_B \colon T_0(B) \to T(B)$ is surjective. Thus, we obtain the following commutative diagram of $\mathcal{G}$-modules with exact rows 
\begin{equation}\label{exactdiagoftori}
\begin{tikzcd}
	1 & {T_1(\AF_{P,\bar{S}})} & {T_0(\AF_{P,\bar{S}})} & {T(\AF_{P,\bar{S}})} & 1 \\
	1 & {T_1(P)} & {T_0(P)} & {T(P)} & 1
	\arrow[from=2-1, to=2-2]
	\arrow[from=1-1, to=1-2]
	\arrow[hook, from=2-2, to=1-2]
	\arrow[hook, from=2-3, to=1-3]
	\arrow[hook, from=2-4, to=1-4]
	\arrow[from=2-2, to=2-3]
	\arrow[from=2-3, to=2-4]
	\arrow[from=2-4, to=2-5]
	\arrow[from=1-4, to=1-5]
	\arrow[from=1-2, to=1-3]
	\arrow[from=1-3, to=1-4]
\end{tikzcd}
\end{equation}
where the vertical maps are the natural diagonal embeddings. Since $\mc{G}$ acts on $\AF_{P,\bar{S}}\simeq \AF_{K,S}\otimes_KP$ through the second factor, and hence $(\AF_{P,\bar{S}})^{\mathcal{G}} = \AF_{K,S}$, by passing to cohomology we obtain the following commutative diagram with exact rows  
\begin{equation}\label{cohdiag}
\begin{tikzcd}
	{T_0(\AF_{K,S})} & {T(\AF_{K,S})} & {H^1(\mc{G},T_1(\AF_{P,\bar{S}}))} \\
	{T_0(K)} & {T(K)} & {H^1(\mc{G},T_1(P))}
	\arrow[from=2-1, to=2-2]
	\arrow["\pi", from=1-1, to=1-2]
	\arrow["\beta", from=1-2, to=1-3]
	\arrow["\alpha", from=2-2, to=2-3]
	\arrow["{\theta_{\bar{S}}}", from=2-3, to=1-3]
	\arrow[hook, from=2-2, to=1-2]
	\arrow[hook, from=2-1, to=1-1]
\end{tikzcd}
\end{equation} 

We now assume that the index $[T_0(\mathbb{A}_{K,S}) : \overline{T_0(K)}^{(S)}]$ is finite and show that 
\begin{equation}\label{E:EEEE1}
[T(\mathbb{A}_{K,S}) : \overline{T(K)}^{(S)}] \ \ \text{divides} \ \ [T_0(\mathbb{A}_{K,S}) : \overline{T_0(K)}^{(S)}] \cdot \psi(\lambda(d)).
\end{equation}
Since $\overline{T_0(K)}^{(S)}$ is an open subgroup of $T_0(\mathbb{A}_{K,S})$ and, by \cite[Proposition 6.5]{PlRa}, the group homomorphism $\pi \colon T_0(\mathbb{A}_{K,S}) \to T(\mathbb{A}_{K,S})$ is open\footnote{In {\it loc. cit.} this fact is proved for number fields but the argument remains valid also for global function fields. For the reader's convenience, we will sketch a proof in the case at hand. First, since $\pi$ admits a $P$-defined section, for almost all $v \in V^K_f$ and $w \vert v$, the sequence $1 \to T_1(\mathcal{O}_{P_w}) \longrightarrow T_0(\mathcal{O}_{P_w}) \stackrel{\pi}{\longrightarrow} T(\mathcal{O}_{P_w}) \to 1$ is exact.  
Second, since $T_1$ is connected, using Lang's theorem and Hensel's lemma, one concludes that for almost all $v$ we have $H^1(\mathcal{G}_w , T_1(\mathcal{O}_{P_w})) = 1$ (cf. \cite[Corollary of Theorem 6.8]{PlRa}). It follows that for almost all $v \in V^K_f$ we have $\pi(T_0(\mathcal{O}_v)) = T(\mathcal{O}_v)$. Finally, since $\pi$ is a quotient map, its differential is surjective, and therefore $\pi \colon T_0(K_v) \to T(K_v)$ is open for all $v$ by the Implicit Function Theorem (cf. \cite[Proposition 3.4]{PlRa2}). Combining these facts, we obtain the openness of $\pi \colon T_0(\mathbb{A}_{K,S}) \to T(\mathbb{A}_{K,S})$.}, the image $\pi(\overline{T_0(K)}^{(S)})$ is an open subgroup of $T(\mathbb{A}_{K,S})$ contained in $\overline{T(K)}^{(S)}$, and therefore 
$$
\overline{T(K)}^{(S)} = T(K) \cdot \pi(\overline{T_0(K)}^{(S)}).
$$
Set $\Omega := T(K) \cdot \pi(T_0(\mathbb{A}_{K,S}))$. Then the index $[\Omega : \overline{T(K)}^{(S)}]$ divides $[T_0(\mathbb{A}_{K,S}) : \overline{T_0(K)}^{(S)}]$, and in order to complete the proof of (\ref{E:EEEE1}) it remains to show that $[T(\mathbb{A}_{K,S}) : \Omega]$ divides $\psi(\lambda(d))$.  Using the exactness of the top row in (\ref{cohdiag}), we see that 
$$
[T(\mathbb{A}_{K,S}) : \Omega] = [\beta(T(\mathbb{A}_{K,S})) : \beta(\Omega)], 
$$
hence divides $[H^1(\mathcal{G} , T_1(\mathbb{A}_{P,\bar{S}})) : \beta(T(K))]$. 
But since $T_0$ is quasi-split, we have $H^1(\mathcal{G},T_0(P)) = 1$ by Hilbert's 90  for quasi-split tori (cf. \cite[Lemma 2.21]{PlRa2}), and consequently $\alpha$ is surjective. Thus, the latter index equals $\vert \coker \theta_{\bar{S}} \vert$, which divides $\psi(\lambda(d))$ according to Corollary \ref{C:theta} applied to $T_1$ (recall that $\dim T_1 \leq \lambda(d)$ and $\psi$ is super-increasing). 

To complete the argument, we will now combine (\ref{E:EEEE1}) with Theorem \ref{T:2}. If $\mathfrak{d}_K(S \cap \mathrm{Spl}(P/K)) > 0$ then by Theorem \ref{T:2} the index $[T_0(\mathbb{A}_{K,S}) : \overline{T_0(K)}^{(S)}]$ is finite, and then by (\ref{E:EEEE1}) the index $[T(\mathbb{A}_{K,S}) : \overline{T(K)}^{(S)}]$ is also finite. In fact, the proof of Theorem \ref{T:2} yields the following explicit estimate: 
$$
[T(\mathbb{A}_{K,S}) : \overline{T(K)}^{(S)}] \leq \mathfrak{d}_K(S \cap \mathrm{Spl}(P/K))^{-d} \cdot \psi(\lambda(d)). 
$$

Now, if $S$ is a tractable set containing a set of the form $V^K_{\infty} \cup (\mathcal{P}(L/K , \mathcal{C}) \setminus \mathcal{P}_0)$ and condition (\ref{technicalconditionarbtori}) holds, then by Theorem \ref{T:2} the index $[T_0(\mathbb{A}_{K,S}) : \overline{T_0(K)}^{(S)}]$ divides $n^d$, so by (\ref{E:EEEE1}) the index $[T(\mathbb{A}_{K,S}) : \overline{T(K)}^{(S)}]$ divides 
\begin{equation}\label{C(d,n)}
\tilde{C}(d , n) := n^d \cdot \psi(\lambda(d)). 
\end{equation}
\hfill\(\Box\)

\begin{remark}
Using the same ingredients but more precise bookkeeping, one can improve the constant $\tilde{C}(d , n)$ to $n^d \cdot \gamma(d)^{\lambda(d)(\gamma(d) - 1)}$. It would be interesting to find the optimal (or at least close to optimal) function $\tilde{C}(d , n)$. 
\end{remark}

\section{The proof of Theorem \ref{T:A}}\label{ASAreductive}

\begin{samepage}
Let $G$ be a reductive algebraic group over a number field $K$ which is an almost direct product of a $K$-torus $T$ and a semi-simple $K$-group $H$. We first prove Theorem \ref{T:A} in the special case where $H$ is assumed to be simply connected. The argument in this case has two major ingredients: the classical criterion for strong approximation (cf. \cite[Theorem 7.12]{PlRa}) and our Theorem \ref{T:1}. The general case is then reduced to the special case by means of a result stating that any reductive $K$-group is a quotient of a reductive $K$-group as in the special case by a quasi-split torus; see Lemma \ref{L:spec_cover}.

\subsection{Special case: $H$ is simply connected.} 
In this case, we have the following more precise statement that does not depend  on the rank of $H$ and the minimal Galois extension $M$ of $K$ over which $H$ becomes an inner form of a split group.

\end{samepage}

\begin{prop}\label{P:simp-con}
Let $G$ be a reductive group over a number field $K$ and suppose that $G = TH$, an almost direct product of a $K$-torus $T$ with the minimal splitting field $P$ and a semi-simple \emph{simply connected} $K$-group $H$.  Then for any infinite subset $S \subset V^K$ the closure $\overline{G(K)}^{(S)}$ is a normal subgroup of $G(\mathbb{A}_{K,S})$ with abelian quotient  $G(\mathbb{A}_{K,S}) / \overline{G(K)}^{(S)}$. 

Furthermore, if $S$ contains $V^K_{\infty}$ and 
$\mathfrak{d}_K(S \cap \mathrm{Spl}(P/K)) > 0$ then $G$ has almost strong approximation with respect to $S$, and if $S$ is  a tractable set of valuations containing a set of the form $V^K_{\infty} \cup (\mathcal{P}(L/K , \mathcal{C}) \setminus \mathcal{P}_0)$, where $\mathcal{P}_0$ is a set of valuations having Dirichlet density zero and there exists $\sigma \in \mathcal{C}$ such that 
$$
%\begin{equation}\label{E:Cond100}
\sigma \vert (P \cap L) = \mathrm{id}_{P \cap L},
%\end{equation}
$$
then the order of  $G(\mathbb{A}_{K,S}) / \overline{G(K)}^{(S)}$ divides $2^{dr} \cdot \tilde{C}(d , n)$ where $d = \dim T$, $r$ is the number of real valuations of $K$,  $n = [L:K]$  and $\tilde{C}(d , n)$ is the function from Theorem \ref{T:1}. 
\end{prop}

%\begin{prop}\label{P:simp-con}
%Let $G$ be a reductive group over a number field $K$, and suppose that $G = TH$, an almost direct %product of a $K$-torus $T$ and a semi-simple simply connected $K$-group $H$.  
%Furthermore, let $S \subset V^K$ be a tractable set of valuations containing a set of the form $V^K_{\infty} \cup (\mathcal{P}(L/K , \mathcal{C}) \setminus \mathcal{P}_0)$, where $\mathcal{P}_0$ is a set of valuations having Dirichlet density zero. Assume that there exists $\sigma \in \mathcal{C}$ such that 
%\begin{equation}\label{E:Cond100}
%\sigma \vert (P \cap L) = \mathrm{id}_{P \cap L},
%\end{equation}
%where $P$ is the minimal splitting field of $T$. Then $\overline{G(K)}^{(S)}$ is a normal subgroup of $G(\mathbb{A}_K(S))$ for which the index $[G(\mathbb{A}_K(S)) : \overline{G(K)}^{(S)}]$ is finite and divides $2^{dr} \cdot \tilde{C}(d , n)$ where $d = \dim T$, $n = [L:K]$, $r$ is the number of all real valuations of $K$, and $\tilde{C}(d , n)$ is the function from Theorem \ref{T:1}. 
%\end{prop}

For the proof, we consider the exact sequence of $K$-groups
$$
1 \to H \longrightarrow G \stackrel{\pi}{\longrightarrow} T' \to 1,
$$
where $T' = G/H$ and $\pi$ is the quotient map. We have  $H = H_1 \times \cdots \times H_{\ell}$, the direct product of $K$-simple groups $H_i$. Since $S$ is infinite, for each $i = 1, \ldots , \ell$ there exists $v_i \in S$ such that $H_i$ is $K_{v_i}$-isotropic (cf. \cite[Theorem 6.7]{PlRa}). Using the criterion for strong approximation (cf. \cite[Theorem 7.12] {PlRa}), we conclude that $H$ has strong approximation with respect to $S$, i.e. $\overline{H(K)}^{(S)} = H(\mathbb{A}_{K,S})$. Then 
$$
[G(\mathbb{A}_{K,S}) , G(\mathbb{A}_{K,S})] \subset H(\mathbb{A}_{K,S}) \subset \overline{G(K)}^{(S)}, 
$$
implying that $\overline{G(K)}^{(S)}$ is a normal subgroup of $G(\mathbb{A}_{K,S})$, and the quotient  $G(\mathbb{A}_{K,S}) / \overline{G(K)}^{(S)}$ is abelian. Furthermore, there is an injective group homomorphism  
$$
G(\mathbb{A}_{K,S}) / \overline{G(K)}^{(S)} \, \simeq \, \pi(G(\mathbb{A}_{K,S})) / \pi(\overline{G(K)}^{(S)}) \, \hookrightarrow \, T'(\mathbb{A}_{K,S}) / \pi(\overline{G(K)}^{(S)}).  
$$
Now, by Theorem \ref{T:1} the index $[T'(\mathbb{A}_{K,S}) : \overline{T'(K)}^{(S)}]$ is finite if $S$ contains $V^K_{\infty}$ and satisfies $\mathfrak{d}_K(S \cap \mathrm{Spl}(P/K)) > 0$ and divides $\tilde{C}(d , n)$ if $S$ is a tractable set as in the statement. So, to complete the proof it is enough to show that the index $[\overline{T'(K)}^{(S)} : \pi(\overline{G(K)}^{(S)})]$ divides $2^{dr}$ as 
$$
%\begin{equation}\label{E:index}
[T'(\mathbb{A}_{K,S}) : \pi(\overline{G(K)}^{(S)})] = [T'(\mathbb{A}_{K,S}) : \overline{T'(K)}^{(S)}] \cdot     
[\overline{T'(K)}^{(S)} : \pi(\overline{G(K)}^{(S)})]. 
%\end{equation}
$$

\begin{lemma}\label{L:index2} 
The index $[T'(K) : \pi(G(K))]$ is finite and divides $2^{dr}$. 
\end{lemma}
\begin{proof}
Let $V^K_r$ be the set of all real valuations of $K$, and let $v \in V^K_r$. It follows from the Implicit Function Theorem that the subgroup $\pi(G(K_v)) \subset T'(K_v)$ is open (cf. \cite[Corollary 3.7]{PlRa2}), hence contains the {\it topological} connected component $T'(K_v)^{\circ}$. On the other hand, it is well-known that $T'$ is isomorphic over $K_v = \mathbb{R}$ to a torus of the form 
$$
(\mathbb{G}_m)^a \times (\mathrm{R}_{\mathbb{C}/\mathbb{R}}(\mathbb{G}_m))^b \times (\mathrm{R}^{(1)}_{\mathbb{C}/\mathbb{R}}(\mathbb{G}_m))^c
$$
for some nonnegative integers $a , b$ and $c$ (cf. \cite[2.2.4]{PlRa2}), and then $[T'(K_v) : T'(K_v)^{\circ}] = 2^a$ divides $2^d$. It follows that for the subgroup 
$$
N := \bigcap_{v \in V^K_r} (T'(K) \cap T'(K_v)^{\circ}), 
$$
the index $[T'(K) : N]$ divides $2^{dr}$, and it remains to show that $N \subset \pi(G(K))$. 

For any field extension $E/K$, we have an exact sequence 
$$
G(E) \stackrel{\pi}{\longrightarrow} T'(E) \stackrel{\delta_E}{\longrightarrow} H^1(E , H).
$$
Now, let $x \in N$. Due to the exactness of the above sequence for $E = K$, we need to show that the cohomology class $\xi := \delta_K(x)$ is trivial. However, for each $v \in V^K_r$, due to the inclusion $T'(K_v)^{\circ} \subset \pi(G(K_v))$ and the definition of $N$, we have that the class $\xi_v := \delta_{K_v}(x)$ is trivial. Thus, $\xi$ lies in the kernel of the map 
$$
H^1(K , H) \stackrel{\gamma_H}{\longrightarrow} \prod_{v \in V^K_r} H^1(K_v , H). 
$$
But according to the Hasse principle for simply connected groups (cf. \cite[Theorem 6.6]{PlRa}), $\gamma_H$ is injective, and hence $\xi$ is trivial, as required. 
\end{proof} 

Since the map $\pi \colon G(\mathbb{A}_{K,S}) \to T'(\mathbb{A}_{K,S})$ is open (cf. \cite[Proposition 6.5]{PlRa}) and $\overline{G(K)}^{(S)}$ contains its kernel $H(\mathbb{A}_{K,S})$, the image $\pi(\overline{G(K)}^{(S)})$ is closed, hence coincides with $\overline{\pi(G(K))}^{(S)}$. 
So,  Lemma~\ref{L:index2} immediately implies that the index $[\overline{T'(K)}^{(S)} : \pi(\overline{G(K)}^{(S)})]$ divides $2^{dr}$ completing the argument.  \hfill $\Box$

%On the other hand, due to (\ref{E:Cond100}), the index $[T'(\mathbb{A}_K(S)) : \overline{T'(K)}^{(S)}]$ divides $\tilde{C}(d , n)$. Now, Proposition \ref{P:simp-con} follows from (\ref{E:index}). 

\subsection{Existence of special covers.} To consider the general case in Theorem \ref{T:A}, we will first establish the existence of special covers of arbitrary reductive groups that enable one to realize every reductive group as a quotient of a reductive group with {\it simply connected} semi-simple part (= commutator subgroup) by a quasi-split torus.   In order to avoid dealing with quotients by non-reduced group schemes, we limit ourselves to characteristic zero. 
\begin{lemma}\label{L:spec_cover}
Let $G$ be a reductive algebraic group over a field $K$ of characteristic zero, and suppose that $G = TH$, an almost direct product of a $K$-torus $T$ and a semi-simple $K$-group $H$. Let $\ell$ be the rank of $G$, and let $M$ be the minimal Galois extension of $K$ over which $H$ becomes an inner form of a $K$-split group. Then there exists an exact sequence of $K$-groups 
\begin{equation}\label{EQ9}
1 \to T_0 \longrightarrow \widetilde{G} \stackrel{\nu}{\longrightarrow} G \to 1
\end{equation}
such that
\begin{enumerate}[label = \rm{(\arabic*)}]
    \item $T_0$ is a quasi-split $K$-torus that becomes split over $M$; 
    \item $\widetilde{G} =\widetilde{T}\widetilde{H}$ is an almost direct product of a $K$-torus $\widetilde{T}$ which is isogenous to $T_0 \times T,$ with $\dim \widetilde{T} = \ell$, and a semi-simple simply connected $K$-group $\widetilde{H}$.  
\end{enumerate}
\end{lemma}
\begin{proof}
Choose a $K$-defined universal cover $\varphi \colon \widetilde{H} \to H$ (cf. \cite[Proposition 2.27]{PlRa2}), set $D = T \times \widetilde{H}$, and consider the $K$-isogeny $\theta \colon D \to G$ obtained by composing the morphism $D \stackrel{\mathrm{id}_T \times \varphi}{\xrightarrow{\hspace{1cm}}} T \times  H$ with the product morphism $T \times H \to G$. Let $\Phi = \ker \theta$, and observe that since the restriction $\theta \vert T$ is injective, the projection to $\widetilde{H}$ identifies $\Phi$ with its image. Now, let $\widetilde{H}_0$  
be the $K$-quasi-split inner twist of $\widetilde{H}$, and let $T_0$ be a maximal $K$-torus of $\widetilde{H}_0$ contained in a $K$-defined Borel subgroup. Then $T_0$ is quasi-split over $K$ (cf. \cite[Expos\'e XXIV, \S3, Proposition 3.13]{SGA3}), becomes split over $M$\footnote{Indeed, it is well-known that an inner form of a split group that becomes quasi-split over an extension of the base field, is actually split over this extension. Thus, $\widetilde{H}_0$ is split over $M$. Then every torus in an $M$-defined Borel subgroup is $M$-split, and in particular, $T_0$ is $M$-split.},  
%(indeed, since $\widetilde{H}_0$ is a quasi-split inner form over $M$, it is $M$-split, so $T_0$, %being a maximal $M$-torus of an $M$-defined Borel subgroup is also $M$-split), 
and $\Phi$ admits a $K$-embedding into $T_0$. Set 
$$
\widetilde{\Phi} = \{ (x , x^{-1}) \in T_0 \times D \, \vert \, x \in \Phi \} \ \ \text{and} \ \ \widetilde{G} = (T_0 \times D)/\widetilde{\Phi}. 
$$
The composite morphism $\widetilde{H} \to D \to \widetilde{G}$
is a $K$-embedding, and we identify $\widetilde{H}$ with the image of this embedding. Furthermore, $\widetilde{G}$ is an almost direct product $\widetilde{T}\widetilde{H}$, where $\widetilde{T}$ is the image of $T_0 \times T \subset T_0 \times D$ in $\widetilde{G}$, and we note that $\dim \widetilde{T} = \ell$. By construction, the composite morphism $T_0 \times D \stackrel{\mathrm{pr}}{\longrightarrow} D \stackrel{\theta}{\longrightarrow} G$ vanishes on $\widetilde{\Phi}$, hence gives rise to a morphism $\nu \colon \widetilde{G} \to G$. Finally, $\ker \nu$ coincides with the image of $T_0 \times \Phi$ in $\widetilde{G}$, which is isomorphic to $T_0$. 
\end{proof}

\begin{remark} 
One can choose an embedding of $\Phi$ into an $M$-split $K$-quasi-split torus $T_0$ in a variety of ways, and for some choices  $\dim T_0$ may be $< \mathrm{rank}\: H$. This will result in  $\widetilde{G} = \widetilde{T} \widetilde{H}$ with $\dim \widetilde{T} < \ell$.
\end{remark}

\subsection{General case.} Let $G = TH$ be a (connected) reductive algebraic group defined over a number field $K$, let $P$ (resp., $M$) be the minimal Galois extension of $K$ over which $T$ splits (resp., $H$ becomes an inner form of a $K$-split group), and set $E = PM$. We now consider the exact sequence (\ref{EQ9}) constructed in Lemma \ref{L:spec_cover}, and let $S \subset V^K$ be any infinite set. 
%Furthermore, let $S \subset V^K$ be a tractable set of valuations containing a set of the form $V^K_{\infty}  \cup (\mathcal{P}(L/K , \mathcal{C}) \setminus \mathcal{P}_0)$, and assume that there exists $\sigma \in \mathcal{C}$ such that (\ref{technicalconditionreductive}) holds. Set 
%$$
%C(\ell, n, r) := 2^{\ell r} \cdot \tilde{C}(\ell , n),
%$$
%where $\ell$ is the rank of $G$, $n = [L :K]$, and $r$ is the number of all real valuations of $K$. Our goal is to show that $\overline{G(K)}^{(S)}$ is a finite index normal subgroup of 
%$G(\mathbb{A}_K(S))$, with the abelian quotient $G(\mathbb{A}_K(S))/\overline{G(K)}^{(S)}$ of order dividing $C(\ell, n, r)$. 
%
%For this, let us consider the exact sequence (\ref{EQ9}) constructed in Lemma \ref{L:spec_cover}. 
Since $T_0$ is $K$-quasi-split, it follows from Hilbert's Theorem 90 that $\nu(\widetilde{G}(F)) = G(F)$ for every field extension $F/K$, and in particular, $\nu(\widetilde{G}(K_v)) = G(K_v)$ for all $v \in V^K \setminus S$. On the other hand, since $T_0$ is connected, we have $\nu(\widetilde{G}(\mathcal{O}_v)) = G(\mathcal{O}_v)$ for almost all $v \in V^K \setminus S$ (cf. \cite[Proposition 6.5]{PlRa}). It follows that 
\begin{equation}\label{E:surj10}
\nu(\widetilde{G}(\mathbb{A}_{K,S})) = G(\mathbb{A}_{K,S}). 
\end{equation}
Combining this with Proposition \ref{P:simp-con}, we obtain the inclusions 
%Now, $\widetilde{G} = \widetilde{T} \widetilde{H}$ where $\widetilde{T}$ is a $K$-torus of dimension $\ell$ (= rank of $G$) which is isogenous to $T_0 \times T$, hence has $E = PM$ as its (minimal) splitting field, and $\widetilde{H}$ is a semi-simple simply connected $K$-group. Since the condition  
%(\ref{technicalconditionreductive}) holds, we can apply Proposition \ref{P:simp-con} to conclude that $\overline{\widetilde{G}(K)}^{(S)}$ is a finite index normal subgroup of $\widetilde{G}(\mathbb{A}_K(S))$, with the abelian quotient $\widetilde{G}(\mathbb{A}_K(S))/ \overline{\widetilde{G}(K)}^{(S)}$ of order dividing $C(\ell, n, r)$. 
%
%On the other hand, we have the inclusions 
$$
[G(\mathbb{A}_{K,S}) , G(\mathbb{A}_{K,S})] = \nu([\widetilde{G}(\mathbb{A}_{K,S}) , \widetilde{G}(\mathbb{A}_{K,S})]) \subset \nu(\overline{\widetilde{G}(K)}^{(S)}) \subset \overline{G(K)}^{(S)},  
$$
which imply that $\overline{G(K)}^{(S)}$ is a normal subgroup of $G(\mathbb{A}_{K,S})$ with abelian quotient. Besides, it follows from (\ref{E:surj10}) that $\nu$ induces a surjective homomorphism of abelian groups 
$$
\widetilde{G}(\mathbb{A}_{K,S})/\overline{\widetilde{G}(K)}^{(S)} \longrightarrow G(\mathbb{A}_{K,S}) / \overline{G(K)}^{(S)}. 
$$
By construction, $\widetilde{G} = \widetilde{T} \widetilde{H}$ where $\widetilde{T}$ is a $K$-torus of dimension $\ell$ (= rank of $G$) which is isogenous to $T_0 \times T$, hence has $E = PM$ as its (minimal) splitting field, and $\widetilde{H}$ is a semi-simple simply connected $K$-group. So, by Proposition \ref{P:simp-con}, the quotient $\widetilde{G}(\mathbb{A}_{K,S})/\overline{\widetilde{G}(K)}^{(S)}$ is finite if $S$ contains $V^K_{\infty}$ and satisfies $\mathfrak{d}_K(S \cap \mathrm{Spl}(E/K)) > 0$. Furthermore, it  has order dividing 
$$
C(\ell, n, r) := 2^{\ell r} \cdot \tilde{C}(\ell , n) 
$$
if $S$ is a tractable set as in the statement of the theorem. Then the same assertions remain valid for the quotient  $G(\mathbb{A}_{K,S}) / \overline{G(K)}^{(S)}$, completing the proof of Theorem \ref{T:A}. \hfill $\Box$

%
%
%and consequently the order of $ G(\mathbb{A}_K(S)) / \overline{G(K)}^{(S)}$ divides $C(\ell, n, r)$, completing the proof of Theorem \ref{T:A}. \hfill $\Box$

%\begin{remark}
%The above argument actually shows that for any (connected) reductive $K$-group $G$ and any {\it infinite} subset $S \subset V^K$, the closure $\overline{G(K)}^{(S)}$ is a normal subgroup of $G(\mathbb{A}_K(S))$ with abelian quotient $G(\mathbb{A}_K(S)) / \overline{G(K)}^{(S)}$.     
%\end{remark}

\section{A counter-example}\label{ASAfailure}

Let now $G$ be a semi-simple algebraic group defined over a number field $K$, and let $M$ be the minimal Galois extension of $K$ over which $G$ becomes an inner twist of a split group. If $G$ is itself an inner form (i.e., $M = K$) then it follows from Theorem \ref{T:A} that $G$ has almost strong approximation for any $S$ containing $V^K_{\infty}$ and satisfying $\mathfrak{d}_K(S \cap V^K_f) > 0$, and in particular for any tractable $S$ (cf. Corollary \ref{C:1}). In the general case, the statement of the theorem includes the stronger assumption that $\mathfrak{d}_K(S 
\cap \mathrm{Spl}(M/K)) >  0$. The goal of this section is to show that this condition cannot be weakened. More precisely, we will construct an example of an absolutely almost simple {\it adjoint}\footnote{As we already noted, it follows from the classical criterion for strong approximation that a simply connected semi-simple group always has strong approximation with respect to any infinite set $S$.} outer form that fails to have almost strong approximation for a suitable tractable set of valuations $S$ (for which $\mathfrak{d}_K(S \cap V^K_f) > 0$ but $\mathfrak{d}_K(S \cap \mathrm{Spl}(M/K)) = 0$).  
%(for which (\ref{technicalconditionsemisimple}) fails to hold). 
%and let $S$ be a tractable set of valuations of $K$ that contains a set of the form $V_\infty^K\cup(\mc{P}(L/K,\mc{C})\setminus\mc{P}_0)$ in our standard notations. According to Theorem \ref{T:A}, the condition that guarantees 
%almost strong approximation in $G$ is  
%\begin{equation}\label{technicalconditionsemisimple}
%\sigma|{(M\cap L)}=\id_{M\cap L} \ \ \text{for some} \ \ \sigma \in \mathcal{C}.
%\end{equation}
%In particular, an inner form of a split group (i.e. when $M = K$) always has almost strong approximation for any tractable set $S$ (cf. Corollary \ref{C:1}). 
%The goal of this section is to show that condition (\ref{technicalconditionsemisimple}) cannot be omitted in the general case. More precisely, we will construct an example of an absolutely almost simple {\it adjoint}\footnote{As we already noted, it follows from the classical criterion for strong approximation that a simply connected semi-simple group always has strong approximation with respect to any infinite set $S$.} outer form that fails to have almost strong approximation for a suitable tractable set of valuations (for which (\ref{technicalconditionsemisimple}) fails to hold). 

\vskip6pt

\rm{Let $K = \mathbb{Q}$, $L = \mathbb{Q}(i)$ and let $\sigma$ be the nontrivial automorphism in $\mathrm{Gal}(L/\QQ)$. For a set $S$ of valuations of $K$, the corresponding ring of $S$-adeles of $K$ will be denoted simply by $\mathbb{A}_S$. Now, fix an \underline{odd} integer $n \geq 3$, and let $q$ denote an arbitrary nondegenerate {$\sigma$-Hermitian form} on $L^n$. Then the  algebraic group $\widetilde{G} = \mathrm{SU}_n(q)$ associated with the special unitary group of $q$ (cf. \cite[Ch. 2, \S2.3.3]{PlRa2}) is an absolutely almost simple simply connected algebraic $\QQ$-group that is an {\it outer form} of type $\textsf{A}_{n-1}$ (cf. \cite[Ch. 2, \S2.1.14]{PlRa2} and \cite[Ch. 12, \S12.3.8]{Springer}), and  $L$ coincides with the minimal Galois extension $M$ of $\QQ$ over which $\widetilde{G}$ becomes an inner form of the split group. We will be working with the corresponding adjoint group $G$ and its central $\QQ$-isogeny $\pi \colon \widetilde{G} \to G$. Note that $F := \ker \pi$ coincides with $\mathrm{R}^{(1)}_{L/\QQ}(\mu_n)$, where $\mu_n$ is the group of $n$-th roots of unity, and will be identified with the $n$-torsion subgroup of the norm torus $T :=\mathrm{R}_{L/\QQ}^{(1)}(\mathbb{G}_m)$.}

Set
$$
S := \{ v_{\infty}\} \cup\{v_2\}\cup \mathbb{P}_{3(4)}.
$$
Observe that the primes $p \in\mathbb{P}_{3(4)}$ correspond precisely to the valuations in the generalized arithmetic progression $\mathcal{P}(L/\QQ , \{\sigma\})$, so $S$ is a tractable set containing $\{v_\infty\}\cup\mc{P}(L/\QQ, \{\sigma \})$. Note that $\mathfrak{d}_{\mathbb{Q}}(S) = 1/2 > 0$ but $S \cap \mathrm{Spl}(L/\mathbb{Q}) = \varnothing$. 
%for which (\ref{technicalconditionsemisimple}) obviously does not hold. 
Our goal is to prove the following.

\begin{thm}\label{noASAexample}
We have $[G(\mathbb{A}_S) : \overline{G(\mathbb{Q})}^{(S)}] = \infty$, and thus $G$ does not have almost strong approximation with respect to $S$.    
\end{thm}
 
\vskip1mm

The proof is based on the observation that for every open subgroup $U$ of $G(\mathbb{A}_S)$, we have the inclusion: 
$$G(\mathbb{Q})\cdot U \supset \overline{G(\mathbb{Q})}^{(S)}.$$  
So, to prove the theorem it suffices to find a sequence of open subgroups $U_1, U_2, \ldots$ such that the products $G(\mathbb{Q})\cdot U_{\ell}$ are all subgroups of $G(\mathbb{A}_S)$ with the indices
\begin{equation}\label{E:X4}
[G(\mathbb{A}_S) : G(\mathbb{Q})\cdot U_{\ell}] \longrightarrow \infty \ \ \text{as} \ \ \ell \longrightarrow \infty.
\end{equation}
The implementation of this idea requires some preparation.

First, it is known that the quasi-split torus $T_0 := \mathrm{R}_{L/\QQ}(\mathbb{G}_m)$ fails to have strong approximation with respect to $S$, and in fact the quotient
$T_0(\mathbb{A}_S)/\overline{T_0(\QQ)}^{(S)}$ has infinite exponent (cf. \cite[Proposition 4]{PrRa-Irr}). For our purposes, we need a statement along these lines for the norm torus $T=\mathrm{R}_{L/\QQ}^{(1)}(\GG_m)$. We have already seen in Example \ref{noASAbctechcond} that $T$ does not have almost strong approximation with respect to $S$ but here we will prove a somewhat stronger statement. To formulate it, 
for every integer $\ell \geq 1$, 
we choose a subset $\Pi_\ell \subset \mathbb{P}_{1(4n)}$ 
of size $\ell$, and then consider the following subgroups of $T(\AF_S)$:
$$
\Gamma(\Pi_\ell) := \prod_{p \in \Pi_\ell} T(\mathbb{Z}_p)\times\prod_{p\in\mathbb{P}_{1(4)}\setminus\Pi_\ell}\{1\},$$ 
$$\Delta(\Pi_\ell) := \prod_{p \in \Pi_\ell} T(\mathbb{Z}_p)^n \times \prod_{p \in \mathbb{P}_{1(4)} \setminus \Pi_\ell} T(\mathbb{Z}_p),
$$
where we consider the standard integral structure on $T$ for which $T(\mathbb{Z}_p)$ is the maximal compact subgroup of $T(\mathbb{Q}_p)$ for all $p$ and $T(\ZZ_p)^n$ denotes the subgroup of $n$th powers in $T(\ZZ_p)$. We note that $\Gamma(\Pi_\ell) \cap \Delta(\Pi_\ell) = \Gamma(\Pi_\ell)^n$ and that the product $\Gamma(\Pi_\ell)\cdot\Delta(\Pi_\ell)$ coincides with $\Delta := \prod_{p \in \mathbb{P}_{1(4)}} T(\mathbb{Z}_p)$.

\vskip1mm
\begin{lemma}\label{Fact1}
%\noindent {\sc Fact 1.} 
{\it We have
$$
i(\Pi_\ell) := [T(\mathbb{A}_S) : T(\mathbb{Q})\cdot \Delta(\Pi_\ell) \cdot T(\mathbb{A}_S)^n] \longrightarrow \infty \ \ \text{as} \ \ \ell \longrightarrow \infty
$$
for any choice of $\Pi_\ell$.}
\end{lemma}
\vskip1mm

\noindent {\it Proof.}  
Since the class number of $L$ is one, we have 
\begin{equation}\label{E:WWW1}
T_0(\mathbb{A}_S) = T_0(\mathbb{Q})\cdot\Delta_0 \ \ \text{where} \ \  \Delta_0 = \prod_{p \in \mathbb{P}_{1(4)}} T_0(\mathbb{Z}_p)
\end{equation}
(cf. \cite[\S 5.1 and 8.1]{PlRa}). 
For each $p \in\mathbb{P}_{1(4)}$, there exists a $\mathbb{Z}_p$-defined isomorphism $T_0 \simeq \mathbb{G}_m\times \mathbb{G}_m$, with $\sigma$ acting by switching the components, and then $T \simeq \{ (t , t^{-1}) \, \vert \, t \in \mathbb{G}_m \}$. It follows that every $t \in T(\mathbb{Q}_p)$ (resp., $\in T(\mathbb{Z}_p)$) can be written in the form $t = \sigma(s)s^{-1}$ for some $s \in T_0(\mathbb{Q}_p)$ (resp., $\in T_0(\mathbb{Z}_p)$), and therefore all elements $t \in T(\mathbb{A}_S)$ are of the form $\sigma(s)s^{-1}$ for some $s \in T_0(\mathbb{A}_S)$. In conjunction with (\ref{E:WWW1}), this yields 
\begin{equation}\label{E:WWW2}
T(\mathbb{A}_S) = T(\mathbb{Q})\cdot\Delta.
\end{equation}
In particular, $T(\mathbb{A}_S)^n = T(\mathbb{Q})^n\cdot\Delta^n$, and since $\Delta(\Pi_\ell) \supset \Delta^n$, we obtain 

$$
i(\Pi_\ell) = [T(\Q)
\cdot\Delta  : T(\mathbb{Q})\cdot\Delta(\Pi_\ell)] = [\Delta : \Delta(\Pi_\ell) \cdot (T(\mathbb{Q}) \cap \Delta)].
$$
We now observe that since $T(\mathbb{Q}_p) = T(\mathbb{Z}_p)$ for all $p \in \mathbb{P}\setminus\mathbb{P}_{1(4)}$, we have 
$$
\Gamma := T(\mathbb{Q}) \cap \Delta = T(\mathbb{Q}) \cap \prod_{p\in \mathbb{P}} T(\mathbb{Z}_p) = T(\mathbb{Z}) = \{ \pm 1, \pm i \}.
$$
As $\Delta = \Gamma(\Pi_\ell) \cdot \Delta(\Pi_\ell)$, we now obtain that 
$$
i(\Pi_\ell) = [\Gamma(\Pi_\ell) : \Gamma(\Pi_\ell) \cap (\Gamma \cdot \Delta(\Pi_\ell))] = \frac{[\Gamma(\Pi_\ell) : \Gamma(\Pi_\ell) \cap \Delta(\Pi_\ell)]}{[\Gamma(\Pi_\ell) \cap (\Gamma \cdot \Delta(\Pi_\ell)) : \Gamma(\Pi_\ell) \cap \Delta(\Pi_\ell)]]}
$$
$$
\geq [\Gamma(\Pi_\ell) : \Gamma(\Pi_\ell)^n] / 4. 
$$
But for any $p \in \mathbb{P}_{1(4n)}$, using a $\mathbb{Z}_p$-isomorphism $T \simeq \mathbb{G}_m$, we have

$$
[T(\mathbb{Z}_p) : T(\mathbb{Z}_p)^n] = [ \mathbb{Z}_p^{\times} : {\mathbb{Z}_p^{\times}}^n] = [\mathbb{F}_p^{\times} : {\mathbb{F}_p^{\times}}^n] = n
$$
%since $p \equiv 1(\mathrm{mod}\: n)$. 
Thus, we get the estimate
$$
i(\Pi_\ell) \geq n^{\ell} / 4,
$$
and our assertion follows. 
\hfill\(\Box\)

\vskip1mm

To prove Theorem \ref{noASAexample}, we will apply a variation of techniques developed in \cite[\S 8.2]{PlRa} to compute class numbers of semi-simple groups of noncompact type. We start with the exact sequence of $\QQ$-groups and $\QQ$-morphisms
\begin{equation}\label{E:X1}
1 \xrightarrow{} F \xrightarrow{} \widetilde{G} \xrightarrow{\pi} G \xrightarrow{} 1,
\end{equation}
which for every extension $P/\mathbb{Q}$ gives rise to the coboundary map $\psi_P \colon G(P) \to H^1(P , F)$. Applying this to $P = \mathbb{Q}_p$ and taking the product over all $p \in \mathbb{P}_{1(4)}$, we obtain an exact sequence 
\begin{equation}\label{prodpsipi}
\prod_{p\in \mathbb{P}_{1(4)}}\widetilde{G}(\QQ_p)\xrightarrow{\Pi = \prod_{p}\pi_{\QQ_p}}\prod_{p\in\mathbb{P}_{1(4)}}G(\QQ_p)\xrightarrow{\Psi = \prod_{p}\psi_{\QQ_p}}\prod_{p\in\mathbb{P}_{1(4)}}H^1(\QQ_p,F).
\end{equation}
We let $\pi_{\mathbb{A}_S}$ and $\psi_{\mathbb{A}_S}$ denote the restrictions of $\Pi$ and $\Psi$ to $\widetilde{G}(\mathbb{A}_S)$ and $G(\mathbb{A}_S)$, respectively. Arguing as in the proof of \cite[Proposition 8.8]{PlRa}, one shows that 
$$\Big(\prod_{p\in\mathbb{P}_{1(4)}}\pi_{\QQ_p}(\widetilde{G}(\QQ_p))\Big)\cap G(\AF_S)=\pi_{\AF_S}(\widetilde{G}(\AF_S)),$$
which implies the exactness of the following sequence of groups and group homomorphisms
\begin{equation}\label{piandpsi}
\widetilde{G}(\AF_S)\xrightarrow{\pi_{\AF_S}}G(\AF_S)\xrightarrow{\psi_{\AF_S}}\prod_{p\in\mathbb{P}_{1(4)}}H^1(\QQ_p,F).
\end{equation}
\begin{lemma}\label{Fact2}
{\it For every open subgroup $U$ of $G(\mathbb{A}_S)$,

\ {\rm (i)} The product $G(\QQ) \cdot U$ is a normal subgroup of $G(\mathbb{A}_S)$;

{\rm (ii)} The map $\psi_{\mathbb{A}_S}$ induces a group  isomorphism of quotients $$G(\mathbb{A}_S)/ (G(\QQ)\cdot U) \simeq \psi_{\mathbb{A}_S}(G(\mathbb{A}_S)) / \psi_{\mathbb{A}_S}(G(\QQ) \cdot U).$$}
\end{lemma}
\vskip1mm 

\noindent {\it Proof.} 
We argue as in the proof of \cite[Proposition 8.8]{PlRa}. 
First, as we already mentioned, since $S$ is infinite,  $\widetilde{G}$ has strong approximation with respect $S$, and therefore $\widetilde{G}(\mathbb{A}_S) = \widetilde{G}(\mathbb{Q}) \cdot \pi^{-1}_{\mathbb{A}_S}(U)$ as $\pi_{\mathbb{A}_S}$ is continuous (cf. \cite[\S 5.1]{PlRa2}). Combined with the exactness of (\ref{piandpsi}), this yields 
\begin{equation}\label{E:WWW3}
\ker \psi_{\mathbb{A}_S} = \pi_{\mathbb{A}_S}(\widetilde{G}(\mathbb{A}_S)) \subset G(\mathbb{Q}) \cdot U. 
\end{equation}
Since $\psi_{\mathbb{A}_S}$ is a group homomorphism from $G(\mathbb{A}_S)$ to a commutative group, its kernel contains the commutator subgroup $[G(\mathbb{A}_S) , G(\mathbb{A}_S)]$. On the other hand, if the product of two subgroups of an abstract group contains the commutator subgroup of the group, the product is actually a normal subgroup. In conjunction with (\ref{E:WWW3}), these observations imply (i), and then (ii) follows from the third isomorphism theorem. 
\hfill\(\Box\)

Next, we will apply a similar argument to the Kummer sequence for $T$:
\begin{equation}\label{E:X2} 
1 \to F \longrightarrow T \stackrel{[n]}{\longrightarrow} T \to 1,
\end{equation}
where $[n]$ denotes the morphism of raising to the $n$th power. Again, for any field extension $P/\mathbb{Q}$, we have the coboundary map $\delta_P \colon T(P) \to H^1(P , F)$ noting that $\ker \delta_P = T(P)^n$. Taking the product over all $p \in \mathbb{P}_{1(4)}$, and restricting the maps to $T(\mathbb{A}_S)$, we obtain an exact sequence similar to (\ref{piandpsi}):  
\begin{equation}\label{E:WWW4}
T(\mathbb{A}_S) \xrightarrow{[n]} T(\mathbb{A}_S) \xrightarrow{\delta_{\mathbb{A}_S}} \prod_{p \in \mathbb{P}_{1(4)}} H^1(\mathbb{Q}_p , F).
\end{equation}

\vskip1mm 

\begin{lemma}\label{Fact3}
{\rm (i)} For any field extension $P/\QQ$, the coboundary map $\delta_P \colon T(P) \to H^1(P , F)$ is surjective.  

\vskip0.7mm 
\ {\rm (ii)} There exists a \underline{finite} set of primes $\Omega$ such that for all $p \in \mathbb{P}\setminus\Omega$ we have 
$$
\psi_{\mathbb{Q}_p}(G(\mathbb{Z}_p)) = 
\delta_{\mathbb{Q}_p}(T(\mathbb{Z}_p)).   
$$
\vskip1mm 
{\rm (iii)} $\psi_{\mathbb{A}_S}(G(\mathbb{A}_S))=\delta_{\mathbb{A}_S}(T(\mathbb{A}_S))$. 
\end{lemma}
\vskip1mm 
\noindent {\it Proof.} (i) In view of the exact sequence 
$$
T(P) \stackrel{\delta_P}{\longrightarrow} H^1(P , F) \stackrel{\omega}{\longrightarrow} H^1(P , T) 
$$
coming from (\ref{E:X2}), it is enough to show that $\im \omega$ is trivial. By \cite[Lemma 2.22]{PlRa}, the group $H^1(P , T)$ is isomorphic to $P^{\times}/N_{L \otimes_{\mathbb{Q}} P/P}((L \otimes_{\mathbb{Q}} P)^{\times})$, and therefore has exponent $\leq 2$. On the other hand, $F$ is a cyclic group of order $n$, so $H^1(P , F)$ is annihilated by $n$. Since by our assumption $n$ is {\it odd}, the triviality of $\im \omega$ follows.  

\vskip1mm

(ii) Let $\QQ^{\mathrm{ur}}_p$ be the maximal unramified extension of $\QQ_p$. It follows from \cite[Proposition 6.4]{PlRa} that there exists a finite subset $\Omega \subset \mathbb{P}$ such that for $p \in \mathbb{P} \setminus \Omega$ 
we have $F(\overline{\QQ}) = F(\QQ^{\mathrm{ur}}_p)$ where $\overline{\QQ}$ denotes some fixed algebraic closure of $\QQ$ and 
$$
\psi_{\QQ_p}(G(\ZZ_p)) = H^1(\QQ^{\mathrm{ur}}_p/\QQ_p , F) = 
\delta_{\QQ_p}(T(\ZZ_p)). 
$$

\vskip1mm 

(iii) For every prime $p$, we have the exact sequence
$$
G(\mathbb{Q}_p) \stackrel{\psi_{\mathbb{Q}_p}}{\longrightarrow} H^1(\mathbb{Q}_p , F) \longrightarrow H^1(\mathbb{Q}_p , \widetilde{G}).
$$ 
Since $\widetilde{G}$ is semi-simple and simply connected, we have $H^1(\mathbb{Q}_p , \widetilde{G}) = 1$ (cf. \cite[Theorem 6.4]{PlRa}), so 
$$
\psi_{\QQ_p}(G(\QQ_p)) = H^1(\QQ_p , F) = \delta_{\QQ_p}(T(\QQ_p))
$$
in view of part (i). In conjunction with part (ii), this yields our claim. \hfill $\Box$

\vskip1mm 

\noindent{\it Proof of Theorem \ref{noASAexample}.}

Let $\Omega$ be the exceptional set of primes in Lemma \ref{Fact3}(ii). For $\ell \geq 1$, pick a subset $\Pi_\ell \subset \mathbb{P}_{1(4n)} \setminus \Omega$ of size $\ell$ and set $\Pi^*_\ell := \Pi_\ell \cup \Omega$. 

We will show that the open subgroups 
$$U_{\ell} := \prod_{p \in \Pi^*_\ell} \pi_{\QQ_p}(\widetilde{G}(\mathbb{Q}_p)) \times \prod_{p \in \mathbb{P}_{1(4)} \setminus \Pi^*_\ell} G(\mathbb{Z}_p) 
$$
satisfy (\ref{E:X4}). Indeed, by construction we have the inclusion $\psi_{\mathbb{A}_S}(U_{\ell}) \subset \delta_{\mathbb{A}_S}(\Delta(\Pi_\ell))$, and it follows from Lemma \ref{Fact3}(i) that $\psi_{\QQ}(G(\QQ)) \subset \delta_{\QQ}(T(\Q))$. According to Lemma \ref{Fact2}, the product $G(\QQ) \cdot U_{\ell}$ is a normal subgroup of $G(\mathbb{A}_S)$, with the quotient isomorphic to $\psi_{\mathbb{A}_S}(G(\mathbb{A}_S))/\psi_{\mathbb{A}_S}(G(\QQ) \cdot U_{\ell})$. The latter admits an epimorphism onto 
$$
\delta_{\mathbb{A}_S}(T(\mathbb{A}_S))/\delta_{\mathbb{A}_S}(T(\QQ) \cdot \Delta(\Pi_\ell)) \simeq T(\mathbb{A}_S)/ T(\QQ) \cdot \Delta(\Pi_\ell) \cdot T(\mathbb{A}_S)^n
$$
(the isomorphism follows from the exact sequence (\ref{E:WWW4})). Thus, 
$$
[G(\mathbb{A}_S) : G(\QQ) \cdot U_{\ell}] \geq i(\Pi_\ell), 
$$
and then (\ref{E:X4}) follows from Lemma \ref{Fact1}. 
\hfill\(\Box\)

\section{The proof of Theorem \ref{T:B}}\label{CSPapplication}

\subsection{Overview of the congruence subgroup problem.} 

Let $G$ be a linear algebraic group defined over a global field $K$. Given a subset $S \subset V^K$ containing $V^K_{\infty}$, we let $\mathcal{O}(S)$ denote the corresponding ring of $S$-integers. First, we fix a faithful $K$-defined representation $\iota \colon G \hookrightarrow \mathrm{GL}_n$, which enables us to speak unambiguously about the group of $\mathcal{O}(S)$-points $\Gamma_S = G(\mathcal{O}(S))$ and its congruence subgroups $\Gamma_S(\mathfrak{a}) = G(\mathcal{O}(S) , \mathfrak{a})$ for (nonzero) ideals $\mathfrak{a} \subset \mathcal{O}(S)$ -- these are defined respectively as 
$\iota^{-1}(\iota(G(K)) \cap \mathrm{GL}_n(\mathcal{O}(S)))$ and $\iota^{-1}(\iota(G(K)) \cap \mathrm{GL}_n(\mathcal{O}(S) , \mathfrak{a}))$, where $$\mathrm{GL}_n(\mathcal{O}(S) , \mathfrak{a}) = \left\{ X \in \mathrm{GL}_n(\mathcal{O}(S)) \, \vert \, X \equiv I_n(\mathrm{mod}\: \mathfrak{a}) \right\}.$$ Then $\Gamma_S(\mathfrak{a})$ is a finite index normal subgroup of $\Gamma_S$ for any nonzero ideal $\mathfrak{a}$, and the {\it Congruence Subgroup Problem} (CSP) for $\Gamma_S$ in its classical formulation is the question of whether every finite index normal subgroup of $\Gamma_S$ contains a suitable congruence subgroup $\Gamma_S(\mathfrak{a})$. We refer the reader to the surveys \cite{PrRa-Milnor} and \cite{RaCSP} for a discussion of various approaches developed to attack 
the CSP, and of the results obtained using these techniques. Here we only recall the reformulation of the CSP suggested by J.-P.~Serre \cite{SerreSL2}.  

Let $\mathcal{N}^S_a$ (resp., $\mathcal{N}^S_c$) be the family of all finite index normal subgroups $N \subset \Gamma_S$ (resp., of all congruence subgroups $\Gamma_S(\mathfrak{a})$ for nonzero ideals $\mathfrak{a} \subset \mathcal{O}(S)$). Then there are topologies $\tau^S_a$ and $\tau^S_c$ (called the $S$-{\it arithmetic} and $S$-{\it congruence} topologies) on the group $G(K)$ that are compatible with the group structure and have $\mathcal{N}^S_a$ and $\mathcal{N}^S_c$ respectively as fundamental systems of neighborhoods of the identity. Furthermore, $G(K)$ admits completions with respect to the uniform structures associated with $\tau^S_a$ and $\tau^S_c$ that will be denoted $\widehat{G}^S$ and $\overline{G}^S$. Since $\tau^S_a$ is apriori stronger than $\tau^S_c$, there exists a continuous group homomorphism $\pi \colon \widehat{G}^S \to \overline{G}^S$, which is, in fact,   surjective. Its kernel $C^S(G) := \ker \pi$ is a profinite group called the $S$-{\it congruence kernel} (one shows that it is independent of the initial choice of a faithful representation $\iota$). Thus, we have the following exact sequence of locally compact topological groups
\begin{equation}\tag{CSP}\label{E:CSP}
1 \to C^S(G) \longrightarrow \widehat{G}^S \stackrel{\pi}{\longrightarrow} \overline{G}^S \to 1. 
\end{equation}    
It is easy to see that the affirmative answer to the classical congruence subgroup problem for $\Gamma$ amounts to the fact that the topologies $\tau^S_a$ and $\tau^S_c$ coincide, which turns out to be equivalent to $C^S(G) =  1 $. In the general case,  $C^S(G)$ measures the difference between the two topologies. So, Serre proposed to reinterpet the CSP as the problem of computation of $C^S(G)$. Furthermore, he formulated the following conjecture that describes $C^S(G)$ in the main case where $G$ is absolutely almost simple and simply connected and $S$ is finite: {\it if $\mathrm{rk}_S\: G := \sum_{v \in S} \mathrm{rk}_{K_v}\: G > 1$ and $\mathrm{rk}_{K_v}\: G > 0$ for all $v \in S \setminus V^K_{\infty}$ then $C^S(G)$ should be finite, and if $\mathrm{rk}_S\: G = 1$ then it should be infinite.} In fact, if the Margulis-Platonov conjecture concerning the normal subgroup structure of $G(K)$ holds - see below, the finiteness of $C^S(G)$ is equivalent to its {\it centrality} (i.e., to the fact that (\ref{E:CSP}) is a central extension), in which case it is isomorphic to the {\it metaplectic kernel} $M(S , G)$ that was computed in \cite{RaPra} in all cases relevant to the CSP. 

In the current paper, we are interested only in the higher rank part of Serre's conjecture which has been confirmed in a number of cases (see \cite{PrRa-Milnor}, \cite{RaCSP}). Nevertheless, it remains completely open for anisotropic groups of types $\textsf{A}_n$ (both inner and outer forms) and $\textsf{E}_6$, triality forms of type $^{3,6}\textsf{D}_4$, and some other situations. Some evidence for Serre's conjecture in these cases has been generated through the investigation of CSP for {\it infinite} $S$. More precisely, the truth of Serre's conjecture combined with  computations of the metaplectic kernel in \cite{RaPra} would imply that $C^S(G) = 1$ for {\it any} infinite $S$ such that $\mathrm{rk}_{K_v}\: G > 0$ for all $v \in S \setminus V^K_{\infty}$, so efforts have been made to prove this for  
certain infinite $S$. In particular, in \cite{PrRa} this was proved for absolutely almost simple simply connected groups of all types when $S$ contains all but finitely many valuations in a generalized arithmetic progression, with an argument not requiring any case-by-case considerations. Subsequently, Radhika and Raghunathan \cite{RadRag} focused on anisotropic inner forms of type $\textsf{A}_n$ (which are all of the form $\mathrm{SL}_{1 , D}$ for some central division $K$-algebra $D$) and extended the result of \cite{PrRa} to a class of sets $S$ which basically coincides with our tractable sets.  In the current paper we extend their result to absolutely almost simple simply connected algebraic groups and all tractable $S$ that warrant the application of our results on  almost strong approximation in tori  
%use our results on almost strong approximation for tori to prove that $C^S(G) = 1$ for all tractable sets $S$ 
-- see Theorem \ref{T:B} for the precise formulation. We note that this formulation includes the Margulis-Platonov conjecture (MP) - see \cite[\S 9.1]{PlRa} for a discussion,  which we now recall for the reader's convenience: 

\vskip2mm 

\noindent {\it Let $G$ be an absolutely almost simple simply connected algebraic group over a global field $K$. Set $\mathcal{A} = \{ v \in V^K_f \, \vert \, \mathrm{rk}_{K_v} \: G = 0 \}$, and let $\delta \colon G(K) \to G_{\mathcal{A}} := \prod_{v \in \mathcal{A}} G(K_v)$ be the diagonal map. Then for every noncentral normal subgroup $N \subset G(K)$ there exists an open normal subgroup $W \subset G_{\mathcal{A}}$ such that $N = \delta^{-1}(W)$. In particular, if $\mathcal{A} = \varnothing$ (which is always the case if the type of $G$ is different from $\textsf{A}_n$), the group $G(K)$ does not have any proper noncentral normal subgroups.}

\vskip2mm 

For results on (MP) obtained prior to 1990 -- see \cite[Ch. IX]{PlRa}. Subsequently, (MP) was proved also for all anisotropic inner forms of type $\textsf{A}_n$ -- see \cite{R-MP}, \cite{Segev}, which explains why no assumption on the truth of (MP) is made in \cite{RadRag}.

\subsection{Proof of Theorem \ref{T:B}}\label{sec6} Our argument will be an adaptation of the proof of Theorem B 
in \cite{PrRa}. We will freely use the notations introduced in the statement of Theorem \ref{T:B}. In particular, $G$ will denote an absolutely almost simple simply connected algebraic group defined over a global field $K$, and $S \subset V^K$ a tractable set of valuations that contains a set of the form $V^K_{\infty}\cup(\mc{P}(L/K , \mathcal{C}) \setminus \mc{P}_0)$ in our standard notations. Furthermore, we let $\mathcal{A} = \{ v \in V^K_f \, \vert \, \mathrm{rk}_{K_v}\: G = 0 \}$ denote the (finite) set of nonarchimedean places where $G$ is anisotropic as in the statement of the Margulis-Platonov conjecture above. We will prove Theorem \ref{T:B} by analyzing the exact sequence (\ref{E:CSP}) written for another set of valuations $\tilde{S}$ such that $V^K_{\infty} \subset \tilde{S} \subset S$ and $S \setminus \tilde{S}$ is finite. First, let $\tilde{S}$ be any such set, and let 
\begin{equation}\label{E:CSP1}
1 \to C^{\tilde{S}}(G) \longrightarrow \widehat{G}^{\tilde{S}} \stackrel{\tilde{\pi}}{\longrightarrow} \overline{G}^{\tilde{S}} \to 1
\end{equation}
be the congruence subgroup sequence (\ref{E:CSP}) for the set $\tilde{S}$. It easily follows from the definitions that the $\tilde{S}$-congruence topology $\tau_c^{\tilde{S}}$ on $G(K)$ coincides with the $\tilde{S}$-adelic topology induced by the embedding $G(K) \hookrightarrow G(\mathbb{A}_{K,\tilde{S}})$. Now, since the set $\tilde{S}$ is infinite, it contains  $v$ such that $G$ is $K_v$-isotropic (cf. \cite[Theorem 6.7]{PlRa}), so $G$ has strong approximation with respect to $\tilde{S}$, implying that the completion $\overline{G}^{\tilde{S}}$ can be identified with $G(\mathbb{A}_{K,\tilde{S}})$. Next, it is enough to prove that (\ref{E:CSP1}) is a central extension for {\it some} $\tilde{S}$ as above. Indeed, using the truth of (MP) and the assumption that 
\begin{equation}\label{E:an} 
\mathcal{A} \cap S = \varnothing, 
\end{equation}
one shows that the congruence kernel $C^{\tilde{S}}(G)$ is isomorphic to (the dual of) the metaplectic kernel $M(\tilde{S} , G)$ (cf. \cite[Proposition 2]{RaCSP}). Again, since $\tilde{S}$ is infinite, it contains a nonarchimedean $v$ such that $G$ is $K_v$-isotropic, and then computations of the metaplectic kernel in \cite[Main Theorem, p. 92]{RaPra} imply that $M(\tilde{S} , G) = 1$, hence $C^{\tilde{S}}(G) = 1$. Since $S \supset \tilde{S} \supset V^K_{\infty}$, there is a natural homomorphism $C^{\tilde{S}}(G) \to C^S(G)$, which because of (\ref{E:an}) is surjective (cf. \cite[Lemma 6.2]{Rag}). 
So, $C^S(G) = 1$, as required. 

In order to choose $\tilde{S}$ and establish the centrality of the corresponding sequence (\ref{E:CSP1}), we will use two statements from \cite{PrRa} whose precise formulations with appropriate references are given below. These formulations required some additional notations which we now introduce. Let $v \in V^K$, and let $T$ be a maximal $K_v$-torus of $G$. We let $T^{\mathrm{reg}}$ denote the Zariski-open subset of $T$ consisting of regular elements, and consider the map 
$$
\varphi_{v , T} \colon G(K_v) \times T^{\mathrm{reg}}(K_v)
\to G(K_v), \ \ (g , t) \mapsto gtg^{-1}.$$ It follows from the Implicit Function Theorem (cf. \cite[Ch. III]{PlRa2}) that $\varphi_{v , T}$ is an open map, and in particular $\mathcal{U}(v , T) := \varphi_{v , T}(G(K_v) \times T^{\mathrm{reg}}(K_v))$ is open in $G(K_v)$. We also note that by construction there are natural maps $G(K) \to \widehat{G}^{\tilde{S}}$ and $G(K) \to \overline{G}^{\tilde{S}}$ (in other words, the exact sequence (\ref{E:CSP1}) splits over $G(K)$). In particular, if $t \in G(K)$ is a regular semi-simple element and $T = Z_G(t)$ is the corresponding torus\footnote{We note that the centralizer $Z_G(t)$ is automatically connected since $G$ is simply connected \cite[Theorem 3.9]{SpSt}.}, we can consider $t$ as an element of both $\widehat{G}^{\tilde{S}}$ and $\overline{G}^{\tilde{S}}$, and then 
$$
\tilde{\pi}(Z_{\widehat{G}^{\tilde{S}}}(t)) \, \subset \, Z_{\overline{G}^{\tilde{S}}}(t) = T(\mathbb{A}_{K,\tilde{S}})
$$
(under the identification of $\overline{G}^{\tilde{S}}$ with $G(\mathbb{A}_{K,\tilde{S}})$). We can now formulate a sufficient condition (in fact, a criterion) for the centrality of (\ref{E:CSP1}).
\begin{thm}\label{centrality}
{\rm (\cite[Theorem 3.1(ii)]{PrRa})} Assume that $G(K)$ satisfies (MP) and that $\mathcal{A} \cap \tilde{S} = \varnothing$, and suppose that there is an integer $m>1$, a finite subset $V\subset V^K\setminus \tilde{S}$, and a maximal $K_v$-torus  $T_v$ of $G$ for each $v\in V$ such that for any element $t\in G(K)\cap \prod_{v\in V}\mc{U}(v, T_v)$ (which is automatically regular semi-simple) and the corresponding torus $T = Z_G(t)$, the following inclusion holds:
\begin{equation}\label{E:YYY10}
T(\AF_{K,\tilde{S}})^m \, \subset \,\tilde{\pi}(Z_{\widehat{G}^{\tilde{S}}}(t)).
\end{equation}
Then (\ref{E:CSP1}) is a central extension.
\end{thm}

The verification of condition (\ref{E:YYY10}) is based on the observation that the image $\tilde{\pi}(Z_{\widehat{G}^{\tilde{S}}}(t))$ of the centralizer contains the closure  $\overline{T(K)}^{(\tilde{S})}$ (see below), which enables us to apply our results on almost strong approximation in tori. In fact, we will use not only the finiteness of the index  $[T(\mathbb{A}_{K,\tilde{S}}) : \overline{T(K)}^{(\tilde{S})}]$ but actually a more precise fact that it divides a number that depends only on the rank of $G$ and the degree of the extension $L/K$ involved in the description of the generalized arithmetic progression, allowing us to choose one number that works for {\it all} relevant maximal tori of $G$. First, we need to choose $\tilde{S}$ appropriately for which we use the following.

%In order to be able to verify condition (\ref{E:YYY10}) using our results on almost strong approximation, we will need to choose $\tilde{S}$ appropriately. This is done using the following statement.  

\begin{lemma}\label{disjointextns}
{\rm (\cite[Lemma 5.5]{PrRa})} Let $G$ be an absolutely almost simple simply connected algebraic group defined over a global field $K$, and let $M$ be the minimal Galois extension of $K$ over which $G$ is an inner form of a $K$-split group. Furthermore, suppose we are given a finite subset $\mathbb{S}\subset V^K$ and a finite Galois extension $L/K$. Then there exists a finite subset $V\subset V^K\setminus\mathbb{S}$ and maximal $K_v$-tori $T_v$ of $G$ for each $v\in V$ such that for any $t \in G(K) \cap \prod_{v \in V} \mathcal{U}(v , T_v)$, the minimal 
splitting field $P_T$ of the corresponding torus $T = Z_G(t)$ satisfies  
$$P_T\cap L=M\cap L.$$
\end{lemma}

We are now in a position to complete the proof of Theorem \ref{T:B}. Let $L/K$ be the Galois extension involved in the description of the generalized arithmetic progression in the statement of the theorem, and let $M/K$ be the minimal Galois extension over which $G$ becomes an inner form of a $K$-split group. Applying Lemma \ref{disjointextns} with $\mathbb{S} = \mathcal{A} \cup V^K_{\infty}$, we find a finite subset $V \subset V^K \setminus \mathbb{S}$ and maximal $K_v$-tori $T_v$ of $G$ for $v \in V$ so that for any $t \in G(K) \cap U$, where $U =  \prod_{v \in V} \mathcal{U}(v , T_v)$, and the torus $T = Z_G(t)$ we have 
\begin{equation}\label{E:YYY20}
P \cap L = M \cap L, 
\end{equation}
where $P$ is the minimal splitting field of $T$. 
Set $\tilde{S} = S \setminus V$ (which is obviously tractable), and let $m = \tilde{C}(d , n)$ (the constant from Theorem \ref{T:1}) with $d = \mathrm{rk}\: G$ and $n = [L : K]$.  We will now show that the assumptions of Theorem \ref{centrality} hold true for this $\tilde{S}$, so the theorem will yield the centrality of (\ref{E:CSP1}), completing the argument. Let $t \in G(K) \cap U$ and $T = Z_G(t)$. Then (\ref{E:YYY20}) holds for the splitting field $P$ of $T$, and therefore $\sigma \vert (P \cap L) = \mathrm{id}_{P \cap L}$ for some $\sigma \in \mathcal{C}$. Applying Theorem \ref{T:1}, we conclude that the index $[T(\mathbb{A}_{K,\tilde{S}}) : \overline{T(K)}^{(\tilde{S})}]$ divides $m$, and consequently, 
\begin{equation}\label{E:YYY60}
T(\mathbb{A}_{K,\tilde{S}})^m \subset \overline{T(K)}^{(\tilde{S})}. 
\end{equation}
On the other hand, since $C^{\tilde{S}}(G)$ is compact, the map $\tilde{\pi}$ is proper, so the image $\tilde{\pi}(Z_{\widehat{G}^{\tilde{S}}}(t))$ is closed in $\overline{G}^{\tilde{S}}$. In view of the obvious inclusion $Z_{\widehat{G}^{\tilde{S}}}(t) \supset T(K)$, we get the inclusion $\overline{T(K)}^{(\tilde{S})} \subset \tilde{\pi}(Z_{\widehat{G}^{\tilde{S}}}(t))$, which in conjunction with  
(\ref{E:YYY60}) verifies (\ref{E:YYY10}) and completes the argument.

\begin{remark}
1. For the reader's convenience (and following the referee's suggestion), we would like to summarize the idea of the (rather technical) proof of Theorem \ref{T:B} and clarify the role of almost strong approximation in the argument. Applying Theorem \ref{centrality}, we actually prove the following: {\it Let $G$ be an absolutely almost simple simply connected algebraic group defined over a global field $K$, and let $S \subset V^K$ be an infinite subset that contains $V^K_{\infty}$ but does not contain any nonarchimedean $v$ such that $G$ is $K_v$-anisotropic. Assume that (MP) holds for $G(K)$. If there exists an integer $N > 0$ such that for a ``generic'' regular semisimple element $t \in G(K)$ the centralizer $T = Z_G(t)$ has almost strong approximation relative to $S$ with the index $[T(\mathbb{A}_{K,S}) : \overline{T(K)}^{(S)}]$ dividing $N$, then $C^S(G) = 1$.} The existence of such an $N$ in our situation follows from Theorem \ref{T:1} where it is shown that  $[T(\mathbb{A}_{K,S}) : \overline{T(K)}^{(S)}]$ divides a constant $\tilde{C}(d , n)$ that depends only on $d = \dim T$ and the degree $n$ of the extension $L/K$ used to define the generalized arithmetic progression. 

2. Recently the first-named author proved in \cite{R-CSP1} that $C^S(G)$ is trivial if $G$ satisfies the standard assumptions of Theorem \ref{T:B} and $S$ is such that $S \cap \mathrm{Spl}(M/K)$ has positive Dirichlet density. The argument uses the results of the current paper on almost strong approximation and the triviality of $C^S(G)$ when $S$ almost contains a generalized arithmetic progression. 
    
\end{remark}

\bibliographystyle{amsplain}

\end{document}